\documentclass[11pt]{article}

\usepackage{amsthm,amsmath,amssymb}

\usepackage{graphicx}

\usepackage[colorlinks=true,citecolor=black,linkcolor=black,urlcolor=blue]{hyperref}

\usepackage{stmaryrd}
\usepackage{textcomp}
\usepackage{enumerate}
\usepackage{setspace}
\usepackage{amssymb}
\usepackage{amsthm}
\usepackage{amsmath}
\usepackage{graphicx}
\usepackage{extarrows}
\usepackage{marvosym}
\usepackage{empheq}
\usepackage{latexsym}
\usepackage{endnotes}
\usepackage{fontenc}

\begin{document}
\author{Julian Sahasrabudhe}
\title{Exponential Patterns in Arithmetic Ramsey Theory }
\maketitle

\newtheorem{theorem}{Theorem}
\newtheorem{definition}[theorem]{Definition}
\newtheorem{lemma}[theorem]{Lemma}
\newtheorem{claim}[theorem]{Claim}
\newtheorem{subclaim}[theorem]{Sub-claim}
\newtheorem{corollary}[theorem]{Corollary}
\newtheorem{problem}[theorem]{Problem}
\newtheorem{fact}[theorem]{Fact}
\newtheorem{observation}[theorem]{Observation}
\newtheorem{conjecture}[theorem]{Conjecture}
\newtheorem{question}[theorem]{Question}
\newtheorem{prop}[theorem]{Proposition}
\newcommand{\Addresses}{{
\bigskip
\footnotesize
\textsc{ Julian Sahasrabudhe, Department of Mathematics, University of Memphis, Memphis Tennessee, USA}\par\nopagebreak
 \texttt{julian.sahasra@gmail.com}
}}

\begin{abstract}
We show that for every finite colouring of the natural numbers there exists $a,b >1$ such that the triple $\{a,b,a^b\}$ is monochromatic.\ We go on to show the partition regularity of a much richer class of patterns involving exponentiation.\ For example, as a corollary to our main theorem, we show that for every $n \in \mathbb{N}$ and for every finite colouring of the natural numbers, we may find a monochromatic set including the integers $x_1,\ldots,x_n >1$; all products of distinct $x_i$; and all ``exponential compositions'' of distinct $x_i$ which respect the order $x_1,\ldots,x_n$. In particular, for every finite colouring of the natural numbers one can find a monochromatic quadruple of the form $\{ a,b,ab,a^b \}$, where $a,b>1$.
\end{abstract}

\section{Introduction}
The study of arithmetic Ramsey theory, broadly speaking, looks to better understand what arithmetic structure is guaranteed within a colour class of an arbitrary finite colouring of the natural numbers. Over one hundred years ago, Schur \cite{Sch} proved what is regarded as the first result in this area. He showed that if one colours the integers with finitely many colours, one can find elements $x,y,z$ satisfying $x + y = z$ that all receive the same colour. Another cornerstone result, published in 1927, is due to van der Waerden \cite{vdW} who showed that every finite colouring of the integers contains arbitrarily long monochromatic progressions. The theory of finite systems of equations that may be solved in a colour class has come to be well understood: Rado \cite{Ra}, in 1933, classified the finite systems of homogeneous linear equations that admit a monochromatic solution in an arbitrary colouring of the natural numbers, and subsequently Hindman and Leader \cite{ImagePartReg}, extended these results to a complete theory for finite linear systems. For the details of this development and many other related results, see the survey of Hindman \cite{HindSurvey}.  
\paragraph{}
Progress on finding collections of non-linear patterns that admit monochromatic solutions has proven to be more difficult. Various results on polynomial patterns have been obtained, including the celebrated ``polynomial van der Waerden theorem'' of Bergelson and Leibman \cite{BL} (see also the elementary proof of Walters \cite{Mark}), which states that for any polynomials $f_1,\ldots,f_d$, $d \in \mathbb{N}$, with integer coefficients and constant term zero, one can find a monochromatic patten of the form $\{a + f_1(b), \ldots, a+ f_d(b)\}$, in a given finite colouring, with $a,b \in \mathbb{N}$. However, many open questions remain, even in the domain of polynomial patterns \cite{PeterAndSarkozy}. For example, it is not known if a finite colouring of the integers admits a monochromatic solution to $x^2 + y^2 = z^2$ \cite{ErdosGraham}. Perhaps even more surprisingly, it is unknown whether a finite colouring of $\mathbb{N}$ must contain a monochromatic quadruple of the form $\{a,b,a+b,ab\}$ \cite{HLS}. However, a recent breakthrough on this problem is due to Moreira \cite{Moreira}, who has shown that every finite colouring of $\mathbb{N}$ admits a monochromatic pattern of the form $\{a,a+b,ab\}$. In what follows, we drop the set braces around the monochromatic patterns. 
\paragraph{}
Infinite linear systems of equations have also been studied. Notably, Hindman \cite{Hind} proved that every finite colouring of the natural numbers admits  positive integers $x_1,x_2,x_3, \ldots $ such that all the finite sums of distinct $x_i$ receive the same colour. Other examples of infinite linear systems that admit monochromatic solutions in an arbitrary finite colouring are known \cite{BHLS}, \cite{Mill}, \cite{Tay}, however it appears that we are far from a classification of such infinite systems of equations \cite{HLS}. 
\paragraph{}
The starting point of this paper is a surprising observation made by Sisto in 2011 \cite{Si}, who showed that an arbitrary 2-colouring of $\mathbb{N}$ admits infinitely many monochromatic triples of the form $a,b,a^b$ (henceforth: exponential triples). In his paper, Sisto went on to ask if one can find monochromatic exponential triples in an arbitrary \emph{finite} colouring of $\mathbb{N}$. Brown \cite{Br}, simplifying and extending the proof of Sisto, gave further examples of exponential, monochromatic patterns that are present in an arbitrary 2-colouring of the positive integers and proved some weaker results for more colours. Our first task will be to answer the original question of Sisto by showing that any \emph{finite} colouring of the positive integers admits $a,b >1$ such that $a,b,a^b$ is monochromatic. 

\begin{theorem} \label{thm:MonoExpTriple} For every finite colouring of $\mathbb{N}$ there exists integers $a,b>1$ such that $a,b,a^b$ is monochromatic.  
\end{theorem}

As the proof of this is much simpler than our main theorem, we devote Section~\ref{sec:MonoExpTriple} to giving a self-contained proof of Theorem~\ref{thm:MonoExpTriple}. We should remark that there is no natural ``density-type'' result for exponential triples, in the sense of Szemer\'{e}di's theorem \cite{Sz} (See \cite{TaoVu}). In fact, for $N \in \mathbb{N}$ the set $[N] \setminus \{ a^i : a \in \mathbb{N}, i\geq 2\}$ is of size $(1-o(1))N$ and contains no exponential triple. 

\paragraph{}
The main result of this paper, Theorem~\ref{MainTheorem}, will be that a certain class of patterns defined by exponentiation and multiplication (which we call FEP-sets) is partition regular. As the statement of the result itself is a little technical, we postpone the presentation of Theorem~\ref{MainTheorem}  and start, instead, by mentioning two more palatable corollaries of our master theorem. The first such corollary of Theorem~\ref{MainTheorem} actually possess many (but not all) of the difficulties of the full theorem. 

\begin{theorem} \label{Theorem:abatimesbatotheb}
For every finite colouring of $\mathbb{N}$ there exists $a,b>1$ such that $a,b,ab,a^b$ is monochromatic.
\end{theorem}
To go further, we need some terminology. If $X$ is a finite set, call a function $f:\mathbb{N} \rightarrow X$ a \emph{finite colouring} of $\mathbb{N}$ and refer to the elements of $X$ as \emph{colours}. For a finite colouring $f$ of $\mathbb{N}$, we say that a set $A \subseteq \mathbb{N}$ is monochromatic if $f$ is constant on $A$. If $\mathcal{A} \subseteq \mathcal{P}(\mathbb{N})$, we say that $\mathcal{A}$ is \emph{partition regular} if for every finite colouring $f$ of $\mathbb{N}$, one can find a monochromatic set contained in $\mathcal{A}$. For $m \in \mathbb{N}$ and integers $x_1,\ldots,x_m$, we define the \emph{finite sum set} and \emph{finite product set generated by} $x_1,\ldots,x_m$ as
\[ FS(x_1,\ldots,x_m) = \left\lbrace \sum_{i \in I} x_i : \emptyset \not= I \subseteq [m] \right\rbrace, \] 
\[ FP(x_1,\ldots,x_m) = \left\lbrace \prod_{i \in I} x_i : \emptyset \not= I \subseteq [m] \right\rbrace,
\] respectively. Define the \emph{finite exponential set generated by} $x_1,\ldots,x_m$ by first defining $FE(x_m) = \{ x_m \}$ and then, inductively assuming that $FE(x_m), \ldots, FE(x_2,\ldots,x_{m})$ have been defined, we define
\[ FE(x_1,\ldots,x_m) = \left\lbrace x_i^{e_{i+1}e_{i+2} \cdots e_m} : i \in [m], e_{i+1} \in FE(x_{i+1},\ldots, x_m) \cup \{1 \}, \ldots, e_m \in FE(x_m)\cup \{1 \} \right\rbrace. \]
Informally speaking, $FE(x_1,\ldots,x_m)$ is the set of all exponential compositions of the integers $x_1,\ldots,x_m$ that ``respect the order'' of  $x_1,\ldots,x_m$. The analogy with finite sum and product sets is made more salient by first considering exponentiation as a binary operation $\star$, defined by $a \star b = a^b$, and then noting that the set $FE(x_1,\ldots,x_m)$ contains all elements of the form (although not exclusively)
\[ x_{i_1} \star \cdots \star x_{i_l} 
,\] where $l \in [m]$, $1 \leq i_1 < \cdots < i_l \leq m$, and the brackets may be placed in any syntactically valid manner. 
\paragraph{}
Using this notation, we state our second, ``easily digestible'' corollary of Theorem~\ref{MainTheorem}, before we go on to give it in full generality.  
\begin{theorem} \label{thm:ProductsAndExponents}
For every $m \in \mathbb{N}$ and every finite colouring of $\mathbb{N}$ there exist integers $x_1,\ldots,x_m > 1$, for which \[ FP(x_1,\ldots,x_m) \cup FE(x_1,\ldots,x_m) \] is monochromatic. 
\end{theorem}
Before coming to our main theorem, we augment our definition of a finite product set. Given a set $S \subseteq [m]$, an arbitrary ``weight'' function $W : \mathbb{N}^{(\leq m)} \rightarrow \mathbb{N}$, on subsets of $\mathbb{N}$ of cardinality at most $m$, we define the set
\[ FP_{S,W}(x_1,\ldots,x_m) = \left\lbrace  \prod_{i \in S} x_i^{p_i} : p_i \in [0, W(\{ x_{i+1}, \ldots, x_{m} \})], \text{ for each } i \in S \right\rbrace,
\] and refer to an element of $FP_{S,W}(x_1,\ldots,x_m)$ as a $W$-\emph{weighted product} of $x_1,\ldots,x_m$ \emph{supported on} $\{x_i\}_{i \in S}$. To illustrate, set $m=4$, $S = [4]$ and define the weight function $W$, on subsets of $\mathbb{N}$ of size at most $4$, by $W(\{x_1,\ldots,x_j\}) = \sum_{i}^j x_i$, for $j \in [4] $. Also set $W(\emptyset) = 10$. In this case, the set $FP_{S,W}(x_1,\ldots,x_m)$ is precisely the set of integers $x_1^ax_2^bx_3^cx_4^d$, where $a \in [0,x_2+x_3+x_4]$, $b \in [0,x_3+x_4]$, $c \in [0,x_4]$ and $d \in [0,10]$. If instead we were to set $S = \{2,3\}$, we would have $FP_{S,W}(x_1,x_2,x_3,x_4) = \{ x_2^ax_3^b : a \in [0,x_3+x_4], b \in [0,x_4] \}$. 
\paragraph{}
Finally, we define the \emph{finite exponential-product} set \emph{generated} by $x_1,\ldots,x_m$, with \emph{weight} $W$ to be the set of integers
\[ FEP_W(x_1,\ldots,x_m) = \left\lbrace \prod_{i \in B} x_i^{e_i} : \emptyset \not= B \subseteq [m], e_i \in FP_{[m] \setminus B,W}(x_{i+1},\ldots,x_m), \text{ for each } i \in B \right\rbrace . 
\] One can think of the above FEP-set as the collection of integers that one can construct from $x_1,\ldots,x_m$ according to the following rule: ``First choose some non-empty $B \subseteq [m]$ as the indexes of the `bases' in the product; then, for each $i \in B$, choose an integer $e_i$ to be a $W$-weighted product of $x_1,\ldots,x_m$ that is supported on the elements $\{x_{i+1},\ldots,x_m\} \setminus \{ x_i :i \in B\}$; then form $\prod_{i \in B} x_i^{e_i}$''.
\paragraph{}
So, for appropriately chosen $W$, the set $FEP_W(a,b,c)$ contains $a,b,c$; the ``height-one'' exponential compositions $a^b,b^c,a^c$; the products $ab,ac,bc,abc$; various ``tower-type'' expressions such as $a^{b^c}$, $a^{b^{c^2}}$, $a^{b^cc}$,$a^{b^{c^c}}, b^{c^2}$; along with many patterns which mix exponentiation and multiplication, for example, $a^bc,a^{b^c}c, (ab)^c$, and $a(b^c)$. It is also worth pointing out that for every $k \in \mathbb{N}$, one may choose $W$ appropriately so that $FEP_W(a,b,c)$ contains the ``exponential progression'' $a,b,a^b,a^{b^2},\ldots,a^{b^k}$.
\paragraph{} The main result of this paper, to whose proof we dedicate Sections~\ref{sec:DefsAndPrelims} and \ref{sec:ProofOfMainThm}, is that the collection of $FEP_W$ sets, $\{ FEP_W(x_1,\ldots,x_m) \}_{x_1,\ldots,x_m \in \mathbb{N}}$ is partition regular, for any $m \in \mathbb{N}$ and any weight function $W$.  
\begin{theorem} \label{MainTheorem}
Let $m \in \mathbb{N}$ and $W : \mathbb{N}^{(\leq m)} \rightarrow \mathbb{N}$ be an arbitrary function. For every finite colouring of $\mathbb{N}$ there exist integers $x_1,\ldots,x_m > 1 $ such that $FEP_W(x_1,\ldots,x_m)$ is monochromatic. 
\end{theorem}
From Theorem~\ref{MainTheorem} we easily deduce Theorem~\ref{thm:ProductsAndExponents} (and hence Theorem~\ref{Theorem:abatimesbatotheb}) by simply choosing $W$ to be a function that is growing sufficiently quickly.  
\paragraph{}
Towards proving the non-partition regularity of certain exponential patterns, in Section~\ref{sec:ClassificationResult} we give examples of colourings which we shall use to show the non-partition regularity of several natural exponential patterns. For example, we show that $a,b,a^b,b^a$ is \emph{not} partition regular. Indeed, using these colourings along with Theorem~\ref{MainTheorem}, we shall deduce the following classification of ``height-one'' exponential systems.
\paragraph{}
Suppose that $m \in \mathbb{N}$, and that $R$ is a binary relation on $[m]\times[m]$. If $x_1,\ldots,x_m \in \mathbb{N}$, we define the set of integers  
\[ R\{x_1,\ldots,x_m\} = \left\lbrace x_1,\ldots,x_m \right\rbrace \cup \left\lbrace x_i^{x_j} : (i,j) \in R \right\rbrace
\] and then define the \emph{exponential pattern with shape} $R$ to be the collection of all such $R\{x_1,\ldots,x_m\}$, with $x_1,\ldots,x_m >1$. For example, if $m = 4$ and $R = \{ (1,2),(2,3),(2,4) \}$ the exponential pattern with shape $R$ is the collection
\[ \left\lbrace \{x_1,x_2,x_3,x_4, x_1^{x_2},x_2^{x_3},x_2^{x_4} \} : x_1,x_2,x_3,x_4 \in \mathbb{N}\setminus \{1\} \right\rbrace .
\] We classify the relations that give rise to partition regular exponential patterns. 

\begin{theorem}  \label{ClassificationOfFirstOrderCompositions}
Let $m \in \mathbb{N}$ and let $R \subseteq [m] \times [m]$ be a binary relation. The exponential pattern with shape $R$ is partition regular if and only if $R$ does not contain a directed cycle. 
\end{theorem}
In Section \ref{sec:consistency} we turn to the relationship between exponentiation and addition in the context of partition regularity. In sharp contrast to the situation with \emph{multiplication} and exponentiation, we show that partition regular patterns involving addition and exponentiation are very limited. In particular, we show that the partition regular patterns $x,y,x+y$, and $a,b,a^b$ are \emph{inconsistent}. This means that an arbitrary colouring of $\mathbb{N}$ need not admit a monochromatic set of the form $x,y,x+y,a,b,a^b$, despite the fact it must have a monochromatic triple $x,y,x+y$ (by Schur's Theorem) and a monochromatic triple of the form $a,b,a^b$ (Theorem~\ref{thm:MonoExpTriple}).

\begin{theorem} \label{thm:Inconsistent}
There exists a finite colouring of $\mathbb{N}$ for which there is no monochromatic set of the form $x,y,x+y,a,b,a^b$ with $x,y,a,b \in \mathbb{N}$.
\end{theorem}
More generally, we say that two partition regular patterns $\mathcal{A},\mathcal{B} \subseteq \mathcal{P}(\mathbb{N})$ are \emph{consistent} if for every finite colouring of $\mathbb{N}$, one can find $A \in \mathcal{A}$ and $B \in \mathcal{B}$ for which $f$ is constant on both $A$, $B$ and $f(A) = f(B)$.
\paragraph{}
So we have the following situation; while $x,y,x+y$ and $u,v,u\cdot v$ are known to be consistent (by a result of Hindman\footnote{Hindman actually showed much more. He showed that, in a given finite colouring, one can find $x_1,x_2,\ldots$ and $y_1,y_2,\ldots$ so that $FS(x_1,x_2,\ldots)$ and $FP(y_1,y_2, \cdots )$ lie in the same colour class.} \cite{HindSumsAndProds}) and our main theorem certainly implies that $a,b,a^b$ and $u , v , u \cdot v$ are consistent, it is \emph{not} true that $x,y,x+y$ and $a,b, a^b $ are consistent.
\paragraph{}
In Section~\ref{sec:Questions}, we outline some further questions and conjectures in this area.

\section{Monochromatic Exponential Triples} \label{sec:MonoExpTriple}
The aim of this section is to give a short, self-contained proof that for every finite colouring of $\mathbb{N}$ one can find $a,b>1$ such that the triple $a,b,a^b$ is monochromatic. The purpose of this is two-fold. First, to present a short solution to Sisto's question and, second, to give a flavour for the style of proof that we adopt in this paper. 
\paragraph{}
We use an infinitary approach throughout this paper despite the fact that we could have used finitary arguments and obtained explicit bounds on the various parameters. This approach streamlines the presentation and makes it more readable. 
\paragraph{}
We shall require van der Waerden's classical theorem: for every $l,k \in \mathbb{N}$, there exists an integer $W_k(l)$ so that every $k$-colouring of $[W_k(l)]$ admits a monochromatic progression of length $l$.
\paragraph{}
Let $k \in \mathbb{N}$, $X$ be a countable set and $f : X \rightarrow [k]$ be a $k$-colouring. Also, for each $M \in \mathbb{N}$, let $f_{M} : X \rightarrow [k]$ be a $k$-colouring. We say that the sequence $\{ f_M \}$ \emph{converges} to $f$ and write $f_M \rightarrow f$ if for every \emph{finite} set $Y \subseteq X$ there exists some $M^* \in \mathbb{N}$ so that for all $M \geq M^*$ we have $f_M(x) = f(x)$ for all $x \in Y$. The following basic fact on sequences of colourings is often referred to as the \emph{compactness property}.

\begin{fact} \label{Compactness} Let $X$ be a countable set, $k \in \mathbb{N}$, and let $f_{M} : X \rightarrow [k]$ be a $k$-colouring, for each $M \in \mathbb{N}$. There exists an increasing sequence $\{M(n)\}_n$ and a colouring $f: X \rightarrow [k]$ so that 
\[ f_{M(n)} \rightarrow f  
\]  as   $n \rightarrow \infty$. \qed
\end{fact} 

\paragraph{}
Given colourings $f_1, \ldots ,f_r : \mathbb{N} \rightarrow [k]$, we call a sequence of colours
$c_1, \ldots ,c_r \in [k]$ \emph{large} for $f_1, \ldots ,f_r$ if, for each $M \in \mathbb{N}$, we can
find a progression $P(M)$ of length $M$ such that $f_i(P(M)) = c_i$, for each $i \in [r]$. The following is an easy consequence of van der Waerden's theorem. 

\begin{lemma} \label{lem:ExtentionLemma}
Let $r,k \in \mathbb{N}$, $f_1,\ldots,f_r,f_{r+1} : \mathbb{N} \rightarrow [k]$ be $k$-colourings, and let $c_1,\ldots,c_r \in [k]$. If $c_1,\ldots,c_r$ is large with respect to $f_1,\ldots,f_r$ then there exists a colour $c_{r+1} \in [k]$ so that $c_1,\ldots,c_{r+1}$ is large with respect to $f_1,\ldots,f_{r+1}$. 
\end{lemma}
\begin{proof}
Let $M \in \mathbb{N}$ be a parameter. Use the fact that $c_1,\ldots,c_r$ is large with respect to $f_1,\ldots,f_r$, to find a progression $P(M)$ of length $W_k(M)$
(the $k$-colour, length-$M$ van der Waerden number), so that $f_i(P(M)) = c_i$, for all $i \in [r]$. We apply van der Waerdens's theorem to the progression $P(M)$ coloured with $f_{r+1}$, to obtain a progression $P'(M) \subseteq P(M)$, of length $M$, which is monochromatic with respect to $f_{r+1}$. Therefore $P'(M)$ is monochromatic with respect to each colouring $f_1,\ldots,f_{r+1}$. Now apply this argument for each $M \in \mathbb{N}$. Some colour is obtained infinitely often as the colour $f_{r+1}(P(M))$. Choose $c_{r+1}$ to be this colour. Thus $c_1,\ldots,c_{r+1}$ is large with respect to $f_1,\ldots,f_{r+1}$.\end{proof}

For $d \in \mathbb{N}$, and a $k$-colouring $f : \mathbb{N} \rightarrow [k]$, we define the colouring $f(d;\ \cdot \ ) :\mathbb{N} \rightarrow [k]$, by $f(d;x) = f\left(2^{d2^x} \right) $, for $x \in \mathbb{N}$. The following lemma is the core of the proof of Theorem~\ref{thm:MonoExpTriple}. We shall refer to the itemized conclusions and hypotheses of our lemmas as ``Conclusions'' and ``Hypothesises'', throughout.  Notice that from Conclusion~\ref{Item:distinctCols} in Lemma~\ref{lem:MainLemMonoExpTriples}, we immediately obtain a contradiction if we choose $r > k$.

\begin{lemma} \label{lem:MainLemMonoExpTriples} For $k \in \mathbb{N}$, suppose that $f : \mathbb{N} \rightarrow [k] $ is a $k$-colouring admitting no monochromatic exponential triple. Then, for any $r \in \mathbb{N}$, we may find colours $c_1,\ldots,c_r \in [k]$, and colourings 
$\tilde{f}_1,\ldots,\tilde{f}_r : \mathbb{N} \rightarrow [k] $ so that the following hold.
\begin{enumerate}
\item The sequence of colours $c_1, \ldots c_r$ is large with respect to $\tilde{f}_1, \ldots \tilde{f}_r$.

\item  For each $i \in [r]$ there exists a sequence of positive integers $\{ d_i(N) \}_N$ for which 
\[   f(d_i(N);\ \cdot\ ) \rightarrow \tilde{f}_i 
\] as $N \rightarrow \infty$.
 
\item  \label{Item:distinctCols} The colours $c_1,\ldots,c_r$ are distinct. 
\end{enumerate}
\end{lemma}

\begin{proof} 
We apply induction on $r$. If $r = 1$ we choose $d_1(N) = 1$, for all $N \in \mathbb{N}$, and set $\tilde{f}_1 = f_1$. We therefore have the trivial convergence $f(1 ;\ \cdot\ ) \rightarrow \tilde{f}_1$. Now, by van der Waerden's theorem, there exists $c_1 \in [k]$ that is large with respect to $\tilde{f}_1$. This proves the lemma for $r=1$.
\paragraph{}
For the induction step, assume we have found colourings $\tilde{f}_1, \ldots ,\tilde{f}_r$ with associated colours $c_1, \ldots ,c_{r}$, which satisfy the statement of the lemma. For $M \in \mathbb{N}$, we use the fact that the sequence $c_1,\ldots,c_r$ is large with respect to $\tilde{f}_1,\ldots ,\tilde{f}_r$ to find a progression $P(M)$, of length $M$ for which $f_i(P(M)) = c_i$, for all $i \in [r]$. Letting $d(M)$ be the common difference of $P(M)$, we pass to a subsequence of $\{d(M)\}_{M}$ so that the colouring $f(d(M); \ \cdot\ )$ tends to a limit. We take this subsequence as our sequence $\{ d_{r+1}(M) \}$ and take the resulting limit point to be $\tilde{f}_r$. Thus we have the convergence $f(d_{r+1}(N);\ \cdot\ ) \rightarrow \tilde{f}_{r+1}$ as $N \rightarrow \infty$, by construction. 
\paragraph{}
We now claim that $\tilde{f}_{r+1}(x) \not\in \{ c_1,\ldots,c_r \}$, for all $x \in \mathbb{N}$. Suppose that $\tilde{f}_{r+1}(x_0) = c_t$, for some $x_0 \in \mathbb{N}$ and $t \in [r]$. Since we have the convergence $f(d_{r+1}(N); \ \cdot \ )  \rightarrow \tilde{f}_{r+1}$, we may choose some sufficiently large $N$ so that $N \geq 2^{x_0}$ and $\tilde{f}_{r+1}(x) = f(d_{r+1}(N);x)$, for all $x \in [x_0]$. Therefore   
\[ c_t = \tilde{f}_{r+1}(x_0) = f(d_{r+1}(N);x_0) =  f\left( 2^{d_{r+1}(N)2^{x_0}} \right).
\] 
Now recall that $d_{r+1}(N)$ is the common difference of a progression $P(N) = \{ a(N) + xd_{r+1}(N) : x \in [0,N-1] \}$ of length $N \geq 2^{x_0}$ with the property that $\tilde{f}_t(P(N)) = c_t$. With this in mind, we choose a sufficiently large integer $N'$ so that $f(d_t(N'); x) = \tilde{f}_t(x) $, for all $x \in \left[ a(N) + d(N)2^{x_0} \right]$. We therefore have 
\[ c_t  =  \tilde{f}_t(a(N)) = f(d_t(N'); a(N)) = f\left( 2^{d_t(N')2^{a(N)}} \right) 
\] and 
\[ c_t = \tilde{f}_{t}\left(a(N) + d_{r+1}(N)2^{x_0} \right) = f\left(d_t(N'); a(N) + d_{r+1}(N)2^{x_0} \right) \]
\[ =  f\left( 2 \star \left( d_t(N') 2 \star \left( a(N) + d_{r+1}(N)2^{x_0} \right) \right) \right). \] However, this yields a contradiction, as we may find a exponential triple, monochromatic in $c_t$, by choosing
\[ a = 2^{ d_t(N')2^{a(N)}} , b = 2^{d_{r+1}(N)2^{x_0}} \]
and then noting that 
\[ a^b =  2 \star \left( d_t(N')2 \star \left( a(N) + d_{r+1}(N)2^{x_0} \right) \right). \]
Therefore, $\tilde{f}_{r+1}(x) \not\in \{ c_1,\ldots,c_r \}$.

\paragraph{}
We complete the proof of the lemma by applying Lemma~\ref{lem:ExtentionLemma} to find a colour $c_{r+1}$ so that the sequence of colours $c_1,\ldots,c_{r+1}$ is large with respect to $\tilde{f}_1,\ldots,\tilde{f}_{r+1}$. As $\tilde{f}_{r+1}(x) \not\in \{ c_1,\ldots,c_r \}$ it follows that $c_{r+1}$ is distinct from $c_1,\ldots,c_r$. \end{proof}

Theorem \ref{thm:MonoExpTriple} now easily follows by assuming that a given $k$-colouring, $k \in \mathbb{N}$, admits no monochromatic exponential triple and then applying the above lemma with the choice of $r > k$.

\section{Definitions and Preliminaries} \label{sec:DefsAndPrelims}

Given integers $0 \leq x \leq y $, define $[x,y] = \{x, x+1, \ldots, y \}$ and $[y] = \{1,\ldots, y\}$. We denote by $\mathbb{N}_0$ the set $\mathbb{N} \cup \{0\}$. If $X$ is a set and $k \in \mathbb{N}$, we let $X^{(k)}$ denote the collection of all subsets of $X$ of cardinality $k$ and let $X^{(\leq k)}$ denote the collection of subsets of cardinality \emph{at most} $k$. Suppose that $n,d \in \mathbb{N}$, $A,B \subseteq \mathbb{Z}^n$, and that $t \in \mathbb{Z}^n$. We define $A + t = \{ a + t : a \in A \}$, $d A = \{ da : a \in A \}$, and $A + B = \{ a + b : a \in A ,b \in B \}$.  
\paragraph{}
For positive integers $n,T$, define an $n$-dimensional $T$\emph{-box} to be a set of the form $[0,T]^d + s \subseteq \mathbb{N}_0^n$, where $s \in \mathbb{N}_0^n$. We refer to $s$ as the \emph{corner} of the $T$ box (it is clear that the corner is uniquely defined). For $d,L \in \mathbb{N}$, define a \emph{grid of common difference} $d$, \emph{length} $L$, and \emph{dimension} $n$, to be a set of the form $ s + d[0,L]^n$, where $s \in \mathbb{N}_0^n$.
\paragraph{}
We also define a \emph{fat grid} of \emph{thickness} $T$, \emph{common difference} $d$, \emph{length} $L$, and \emph{dimension} $n$ to be a subset of $\mathbb{N}_0^n$ of the form $s + [0,T]^n + d[0,L]^n$, where $s \in \mathbb{N}^n_0$ and $d > T$. If $G$ is a fat grid, as above, we denote by $B(G)$ the box $s + [0,T]^n $ ($B(G)$ is uniquely determined as $d > T$). Finally, we shall define the \emph{corner} of the fat grid $G$ to be $s$. Note the corner of a gird $G$ coincides with the corner of its box $B(G)$. 

\paragraph{}

In what follows, we shall make repeated use of the following theorem of Gallai \cite{GRS},\cite{Ra}, which gives a natural generalization of van der Waerden's theorem \cite{vdW} to higher dimensions. 
\begin{theorem} \label{thm:Gallai} Let $n \in \mathbb{N}$, every finite colouring of $\mathbb{N}^n$ admits arbitrary long monochromatic grids of dimension $n$. 
\end{theorem}

We shall need another well-known theorem that lies in the canon of classical Arithmetic Ramsey Theory: the Folkman-Sanders ``Finite Sums Theorem'',  \cite{GRS},\cite{Sanders},\cite{Ra'}.  

\begin{theorem} \label{thm:finiteSumsThm} For every $k,m \in \mathbb{N}$ there exists an integer $N(k,m)$ such that every $k$-colouring of $[N(k,m)]$ admits $x_1,\ldots,x_m \in [N(k,m)]$ for which $FS(x_1,\ldots,x_m)$ is monochromatic. 
\end{theorem}

We slightly refine some of the notation introduced in the previous section. Let $X$ be a countable set, $k \in \mathbb{N}$, and let $f_M :X \rightarrow [k]$ be a $k$-colouring, for each $M \in \mathbb{N}$. If $f_M$ tends to a limit $f : X \rightarrow [k]$ we write 
\[ \lim_M f_M  = f 
\] and if $S = \{x_1 < x_2 < \cdots \} \subseteq \mathbb{N}$ define
\[ \lim_{M \in S} f_M = \lim_{n} f_{x_n}, \]
provided the limit on the right-hand side exists.
\paragraph{}
The following is a basic fact about such sequences of colourings. The first conclusion is the a restatement of the \emph{compactness principle} for sequences of finite colourings, as we have seen before, while the latter conclusion is merely an iterated version of the compactness principle, where the colourings are indexed by a multidimensional array. 

\begin{fact} \label{Fact:Compactness}  Let $n,k \in \mathbb{N}$; the following statements hold. 
\begin{enumerate}
\item Every sequence $\{f_M\}_{M}$ of finite colourings $f_M : \mathbb{N}^n \rightarrow [k]$ admits a convergent subsequence. That is, there is a set $S\subseteq \mathbb{N}$ for which the limit 
\[ \lim_{M \in S } f_M, \]  exists.

\item Suppose $m \in \mathbb{N}$, $C_1,\ldots,C_m \subseteq \mathbb{N}$ are infinite sets, and that for each $M_1 \in C_1 , \ldots, M_m \in C_m$ we have a $k$-colouring $f_{M_1,\ldots ,M_m} : \mathbb{N}^n \rightarrow [k]$, then there exist infinite sets $C'_1 \subseteq C_1 ,\ldots, C'_m \subseteq C_m$ so that each of the limits, 
\[ \lim_{M_t \in C'_t} \lim_{M_{t-1} \in C'_{t-1} } \cdots \lim_{M_1 \in C'_1} \ f_{M_1, \ldots, M_m}
\] exist, for each $t \in [r]$ and every choice of $M_{t+1} \in C_{t+1} , \ldots, M_r \in C_r$, if $t < r$. 
\end{enumerate} 
\end{fact}
\begin{proof}
We prove the second item by an iterative application of the first item along with a ``diagonlization'' argument. We show, applying induction on $r \in [0,m]$, that we may find sets $C'_1,\ldots,C'_r \subseteq \mathbb{N}$ so that each of the limits  
\[ \lim_{M_t \in C'_t} \cdots \lim_{M_1 \in C'_1} f_{M_1,\ldots,M_m} \] exist, for any $t \in [0,r]$ and $(M_{t+1},\ldots, M_m) \in C_{t+1} \times \cdots \times C_m$. The basis step $r = 0 $ is trivial, so we assume the lemma holds for $r-1\geq 0$ and show that it holds for $r$. Put 
\[ \widetilde{f}_{M_r,\ldots,M_m} = \lim_{M_{r-1} \in C'_{r-1}} \cdots \lim_{M_1 \in C'_1} f_{M_1,\ldots ,M_m }, 
\] to emphasize that the limiting colouring depends only on $M_r,\ldots,M_m$. If $r = m$, we simply apply the Part 1 of the Fact to the colouring $\widetilde{f}_{M_r}$ to find a set $C'_m$ for which the limit exists. Otherwise assume that $r < m$ and enumerate the set 
$C_{r+1} \times \cdots \times C_m = \{ x_n \}_{n \in \mathbb{N}}$. Apply Part 1 of the Fact, iteratively, to construct subsets of $C_r$, $ C(x_1) \supseteq C(x_2) \supseteq \cdots $, with the property that 
\[ \lim_{M_r \in C(x_n)} \widetilde{f}_{M_r,x_n}
 \]exists for each tuple $x_n$ in the enumeration. Now define a ``diagonal'' set $C'_r = \{y_n \}$, by choosing $y_n \in C(x_1) \cap \cdots \cap C(x_n) \not= \emptyset$ to be such that $y_n > y_{n-1}$. This set $C'_r$ satisfies the induction hypothesis and thus we are done, by induction.
\end{proof}

Throughout the rest of this paper, we shall be interested in colourings $f : \mathbb{N}^n \times \mathbb{N} \rightarrow [k] $, where we think of the first $n$ coordinates as grouped together. These coordinates will receive most of the attention, while the $(n+1)$st coordinate will enjoy a somewhat different role. We refer to such colourings as \emph{partitioned colourings}, and write $f: \mathbb{N}^n \times \mathbb{N} \rightarrow [k] $ and $f(x;y)$, where $x \in \mathbb{N}^n, y \in \mathbb{N}$, to emphasize the separation of the coordinates. As the role of these colourings is ubiquitous in our proof of Theorem~\ref{MainTheorem}, we shall not always explicitly call these colourings ``partitioned''. 

\paragraph{}
Given a finite colouring $f : \mathbb{N}^n \rightarrow [k]$, we say that a fat grid $G = s + [0,T] + d[0,L]^n$ (as above) is \emph{consistent with respect to} $f$ (or merely \emph{consistent}, if $f$ is clear from context) if the colouring $f$ is identical along each translate of $B(G) = s+ [0,T]^n$ that comprises $G$. That is,
\[ f(x) = f(x + vd) 
\] for all $x \in B(G)$ and $v \in [0,L]^n$.
\paragraph{}
Note that it is easy to see that every finite colouring $f : \mathbb{N}^n \rightarrow [k]$ admits arbitrarily long and arbitrarily thick, fat grids of dimension $n$ that are consistent with respect to $f$. Simply apply Gallai's theorem to the product colouring 
\[ F_T(x) = \left( f(x + v) : v \in [0,T]^n \right), 
\] to obtain a monochromatic grid of length $T$, for all $T \in \mathbb{N}$. 
\paragraph{}
For $N \in \mathbb{N}$, and a partitioned colouring $f : \mathbb{N}^n \times \mathbb{N} \rightarrow [k]$, we say that an $n$-dimensional, fat grid of length and thickness $N$ is \emph{consistent} with respect to $f$ if it is consistent with respect to each colouring $f( \ \cdot \ ; y)$, for $y\in [0,N]$. Similarly to our observation above, one can see that a finite colouring $f$ admits arbitrarily long fat grids that are consistent with respect to $f$.
\paragraph{}
We now introduce an object that will help us keep track of certain large-scale information about a colouring in the course of our arguments. Suppose that $f : \mathbb{N}^n \times \mathbb{N} \rightarrow [k]$ is a partitioned $k$-colouring. We say that $\partial f : \mathbb{N}_0^n \times \mathbb{N} \rightarrow [k] $ is a \emph{derived colouring} of $f$ if for every $N \in \mathbb{N}$, there exists a $n$-dimensional fat grid, of length and thickness $N$ that is consistent with respect to $f$ and is such that, for each $y \in [0,N]$, the colouring $f( \ \cdot \ ; y)$ along $B(G)$ is identical to the colouring $\partial f( \ \cdot \ ; y) $ on $[0,N]^n$. In other words, if we write $G(N) = s + [0,N]^n + d[0,N]^n$ ($s \in \mathbb{N}^n_0$, $d \in \mathbb{N}$), we have
\[ f\left(s + x + vd  ; y \right) = \partial f\left( x ;y \right) , 
\] for all $y \in [0,N]$ and $v,x \in [0,N]^n$. Additionally, we define $\partial^0 f = f$ and, for $l\in \mathbb{N}$, we use the notation $\partial^lf$ to represent a derived colouring of $\partial^{l-1}f$.  
\paragraph{}
Another important notion for us is that of a \emph{synchronized array} of colourings. Suppose that $n,r,l \in \mathbb{N}$, that 
$f_1,\ldots,f_r : \mathbb{N}^n  \times \mathbb{N} \rightarrow [k]$ are partitioned $k$-colourings, and that $\partial^if_j : \mathbb{N}^n \times \mathbb{N}\rightarrow [k]$, $i\in [l]$, $j \in [r]$, are derived colourings. Further assume that $d(M) \in \mathbb{N}$, and $s(M) \in \mathbb{N}_0^n$, for each $M \in \mathbb{N}$. We refer to the array of colourings 

\[ f_1, \partial f_1 , \ldots, \partial^lf_1 \]
\[ \vdots \]
\[ f_r, \partial f_r, \ldots , \partial^l f_r 
\] as a \emph{synchronized array} with \emph{profile} $(d(M), s(M))$ if, for every $M \in \mathbb{N}$, there exists a fat grid $G(M)$ such that the following conditions hold.
\begin{enumerate}
\item \label{Def:synArrayconsistent} $G(M)$ is consistent with respect to each of the partitioned colourings $\partial^i f_j $ , $i \in [l], j\in [r]$;
\item \label{Def:synArrayDimension} $G(M)$ has dimension $n$;
\item \label{Def:synArrayLenDiff} $G(M)$ has length and thickness $M$;
\item \label{Def:commonDiff} $G(M)$ has common difference $d(M)$ and corner $s(M)$;
\item \label{Def:SynArrayAlingment} for each $i \in [l-1]$ and $j \in [r]$, the colouring $\partial^if_j( \ \cdot \ ; y)$ along the $M$-box $B(G(M))$ agrees with the colouring $\partial^{i+1}f_j( \ \cdot \ , y)$ along $[0,M]^n$. That is,
\[ \partial^if_j\left(s(M) + x + d(M)v ; y \right) = \partial^{i+1}f_j\left(x ; y \right)
\] for all $y \in [0,M]$, and $v, x \in [0,M]^n $.
\end{enumerate}
We call a grid $G(M)$ that satisfies Conditions (1)-(5) in the above, $M$-\emph{synchronizing for the array} 
$\left( \partial^i f_{i} \right)_{i,j}$. We record some key facts about synchronized arrays. The following lemma allows us to incrementally construct synchronized arrays. This will be crucial in our arguments as we build large collections of subsequences of $\mathbb{N}$ that interact in a favourable way. 
\paragraph{} For the proof of the following Lemma, it is convenient to introduce another fragment of notation. For a partitioned colouring $f : \mathbb{N}^n \times \mathbb{N} \rightarrow [k]$ and $s \in \mathbb{N}_0^n$, let $T_sf : \mathbb{N}^n \times \mathbb{N}\rightarrow [k]$ denote the partitioned colouring defined by $g(x;y) = f(x + s;y)$.
\begin{lemma} \label{synchronizedArrayFact}
\begin{enumerate} Let $n,m,l,r,k \in \mathbb{N}$. We have the following properties of synchronized arrays.
\item If $f : \mathbb{N}^n \times \mathbb{N} \rightarrow [k]$ is a partitioned $k$-colouring then there exist derived, partitioned colourings $\partial f,  \ldots, \partial^l f : \mathbb{N}^n \times \mathbb{N} \rightarrow [k]$, so that $(f, \partial f , \partial^2 f,\ldots,\partial^l f ) $ is a synchronized array.
\item Let $f_1,\ldots,f_r$ be partitioned $k$-colourings of $\mathbb{N}^n \times \mathbb{N}$ and let $\partial f_i, \ldots, \partial^l f_i$, $i \in [r]$, represent derived colourings, such that 
\[ \left( f_i, \partial f_i, \ldots, \partial^l f_i \right)_{i \in [r]} 
\] is a synchronized array. Then, given $m$ new colourings $f_{r+1}, \ldots, f_{r+m} :\mathbb{N}^n \times \mathbb{N} \rightarrow [k]$, we may find derived colourings $\partial f_{r+i}, \ldots, \partial^l f_{r+i}$, for each $i \in [m]$, such that the array 
\[ \left( f_{i}, \partial f_{i}, \ldots, \partial^l f_{i} \right)_{i \in [r+m]}
\] is synchronized. 

\item \label{con:StillsynAftersub} If $\left( f_i, \partial f_i , \ldots, \partial^lf_i \right)_{i \in [r]}$ is a synchronized array with profile $(d(M),s(M))_{M \in\mathbb{N}}$, and if $\left( M(n) \right)_{n\in \mathbb{N}}$ is an increasing sequence of integers, then $\left( f_i, \partial f_i , \ldots, \partial^lf_i \right)_{i \in [r]}$ also has profile \\ $(d(M(n)),s(M(n)))_{n \in \mathbb{N}}$.

\end{enumerate}
\end{lemma}
\begin{proof}
Note that the proof of Conclusion~\ref{con:StillsynAftersub} follows immediately from the definition of a synchronized array. To prove the first two parts of the theorem we iteratively apply Gallai's theorem (Theorem \ref{thm:Gallai}) to an appropriate product colourings and then ``take limits''. Recall that the finitary form of Gallai's theorem tells us that if $l,d,c \in \mathbb{N}$, there exists a smallest integer $N = N(l,d,c)$ so that every $c$-colouring of a $n$-dimensional grid of length $N$ and dimension $d$ admits a monochromatic sub-grid of length $l$ and dimension $d$.  
\paragraph{} For the basis step, let $f: \mathbb{N}^n\times \mathbb{N} \rightarrow [k]$, be a partitioned colouring. We need to construct, for each $M \in \mathbb{N}$, a sequence of grids of length and thickness $M$, that are consistent with respect to each of the colourings $f(\ \cdot\ ; 0) , f(\ \cdot\ ; 1), \ldots , f(\ \cdot \; M)$. To this end, let $M \in \mathbb{N}$ and define the ``product colouring'' $F_M: \mathbb{N}^n \rightarrow [k]^{(M+1)^{n+1}}$ by \[ F_M(z) = \left( f(z+x;y) : y \in [0,M],x \in [0,M] \right) 
\] Now apply Gallai's theorem to this colouring to obtain $s(M) \in \mathbb{N}_0^n$ and $d(M) \in \mathbb{N}$ so that the fat grid $G(M) = s(M) + [0,M] + d(M)[0,M]^n$ is consistent with respect to $F_M$. But we are done, as this is equivalent to $G(M)$ being consistent with respect to each $f(\ \cdot\ ; 0) , f(\ \cdot\ ; 1), \ldots , f(\ \cdot\ ; M)$. Hence, for $l = 0 $, we have constructed a synchronized array. 
\paragraph{}
Now, moving to the induction step, assume that for $l \geq 0$ we have a synchronized array $f,\partial f,\ldots, \partial^lf$. This means that for each $M \in \mathbb{N}$, there exists some fat grid $G(M) = s(M) + [0,M] + d(M)[0,M]^n$ ($s(M) \in \mathbb{N}^n_0$ and $d(M) \in \mathbb{N}$) with length and thickness $M$ that is $M$-synchronizing for this array. We consider the set of ``shifted'' partitioned colourings $\left\lbrace T_{s(M)}\partial^l f \right\rbrace_{M \in \mathbb{N}}$. By the compactness principle, there exists an infinite set $S \subseteq \mathbb{N}$ for which the limit
\[ \lim_{M\in S} T_{s(M)}\partial^l f, 
\] exists. Note that this limit point is a derived colouring of $\partial^l f$ and therefore we may denote it by $\partial^{l+1}f$. 
\paragraph{} Now let us redefine $G'(M) = G(S(M))$, where $S = \{ S(1) < S(2) < \cdots \}$, and note that $G'(M)$ can be thought of as a fat grid with length and thickness $M$, by possibly disregarding some terms, as we must have $S(M) \geq M$. Therefore the grids $G'(M)$ are synchronizing with respect to 
$f,\partial f ,\ldots ,\partial^l f$, by Conclusion~\ref{con:StillsynAftersub} of the Lemma. Also, the grids $G'(M)$ witness conditions (\ref{Def:synArrayconsistent}), (\ref{Def:synArrayDimension}), (\ref{Def:synArrayLenDiff}) in the definition of a synchronized array, with respect to the array $f,\partial f ,\ldots, \partial^{l+1} f $. All that remains is to refine these grids $G'(M)$ so that they are consistent with respect to $\partial^{l+1}f$. This is achieved in a manner similar to the basis step. For each $M \in \mathbb{N}$, set $N  = N\left(M,n,k^{(M+1)^{n+1}}\right)$ and consider the grid $G'(N)$. We use Gallai's theorem to pass to a subgrid of $G'(N)$ that is consistent with respect to $\partial^{l+1}f$. We write $G'(N) = d[0,N]^n + [0,N]^n + s$ and define the product colouring of the $n$-dimensional grid $F : [0,N]^n  \rightarrow [k]^{(M+1)^{n+1}}$ by setting
\[ F(z) = \left( \partial^lf(dz + x + s ;y): x \in [0,M]^n,y \in [0,M]\right), 
\] for each $z \in [0,N]^n$. Applying Gallai's theorem to the colouring $F$, we obtain a sub-grid $G'' = G''(M) \subseteq G'$ of length $M$, for which $F(z)$ is monochromatic. Hence, $G''(M)$ is consistent with respect to all colourings $f,\partial f ,\ldots, \partial^{l+1} f$. Applying this argument for all $M \in \mathbb{N}$ proves the first Conclusion the Lemma. 

\paragraph{}
The proof of the second Conclusion of the lemma is of a very similar nature to the proof of the first and so we don't include it here. 
\end{proof}

The following is a technical fact that makes explicit the observation that nothing essential changes when we append a new ``dead'' variable to the functions $f_i(x_1,\ldots,x_n;y)$ in a synchronized array.
\begin{fact} \label{adjoinADeadVariable}
For $r \in \mathbb{N}$, let $\left( f_i, \partial f_i , \ldots, \partial^lf_i \right)_{i \in [r]}$ be a synchronized array with profile $(d(M),s(M))_{M \in\mathbb{N}}$. Define new functions $\tilde{f_i},\widetilde{\partial^{j}f}$, for $i \in [r], j \in [l]$, as follows,
\[ \tilde{f_i}(x_1,\ldots,x_n, x_{n+1};y) = f_i(x_1,\ldots,x_n;y); \]
\[ \widetilde{\partial^{j}f_i}(x_1,\ldots,x_n,x_{n+1};y) = \partial^jf_i(x_1,\ldots,x_n;y), 
\]for each $(x_1,\ldots,x_{n+1};y) \in \mathbb{N}^{n+1}_0$. Then the following assertions hold
\begin{enumerate}
\item $\widetilde{\partial f_i}$ is a derived colouring of $\widetilde{f_i}$ and hence we may write 
$\partial \tilde{ f}_j = \widetilde{\partial f_j}.$ 
\item  $\left( \tilde{f_i}, \widetilde{\partial f_i} , \ldots, \widetilde{\partial^l f_i} \right)_{i \in [r]}$ is a synchronized array with the same profile, $(d(M),s(M))_{M \in\mathbb{N}}$, as the original array. 
\end{enumerate}
\end{fact}\qed

\section{Proof of Theorem \ref{MainTheorem}} \label{sec:ProofOfMainThm}
The proof of Theorem~\ref{MainTheorem} is broken into three main parts: Lemma~\ref{BasicRamseyLemma}, Lemma~\ref{structureBuildingLemma} and the ``Proof of Theorem~\ref{MainTheorem}''. The core of the argument lies in Lemma~\ref{structureBuildingLemma}, where we build up a large collection of integers $X$, with the property that ``many'' pairs $a,b \in X$ are such that the pair $a,a^b$ is monochromatic.\footnote{We actually need more complicated exponential compositions, but this is the sort of thing we are looking for.} This set of ``good'' pairs will have a sufficiently rich structure that we can apply a Ramsey-type theorem (Lemma~\ref{BasicRamseyLemma}) to find a monochromatic finite product set among these ``good'' pairs. We then observe that the rich structure of these pairs will allow us to extend our finite product set to a finite exponential-product set. In the final, rather technical stages of the proof, we simply need to check that all of our choices work. So, while the ``Proof of Theorem~\ref{MainTheorem}'' is rather drawn out, there is nothing new happening here. We are only combining the ingredients, taking apart our system and checking that our previous choices yield the desired results. 
\paragraph{}
Given $m,n \in \mathbb{N}$ and $A_1,\ldots,A_m \subseteq [n]$, set   
\[ FU(A_1,\ldots,A_m) = \left\lbrace \bigcup_{i \in I} A_i : \emptyset \not= I \subseteq [n] \right\rbrace. \] It is a well known fact (for example, as a consequence of the Hales-Jewett theorem \cite{HalesJewett}) that, for sufficiently large $N$, every finite colouring of $\mathcal{P}([N])$ admits a monochromatic set of the form $FU(A_1,\ldots,A_m)$. We require a slight strengthening of this well-known result. As usual, we define the sum $x_i + x_{i+1} + \cdots + x_r$ to be $0$ if we have $i > r$.
\begin{lemma} \label{BasicRamseyLemma} Let $m,k \in \mathbb{N}$, then there exists a positive integer $N = N(k,m)$ so that every $k$-colouring of $\mathcal{P}\left( [N] \right)$ admits $A_1,\ldots,A_m \subseteq [N]$ such that $FU(A_1,\ldots,A_m)$ is monochromatic and $\max{A_i} < \min{A_{i+1}}$, for each $i \in [m-1]$. 
\end{lemma}
\begin{proof}
For $r,t,k \in \mathbb{N}$, let $R_r(t) = R_r(k;t)$ denote the minimum integer $n$, so that every $k$-colouring of $[n]^{(r)}$ admits a set $X \subseteq [n] $, of cardinality $t$ with $X^{(r)}$ monochromatic. Let $S = S(m,k)$ denote the minimum integer $n$ for which every $k$-colouring of $[n]$ admits a monochromatic set of the form $FS(a_1,\ldots,a_m)$. We claim that the choice of $N  = R_1(R_2(R_3( \cdots R_S(mS))\cdots) $ suffices for the lemma. 
\paragraph{}
To see this, let $F : \mathcal{P}([N]) \rightarrow [k]$ be a $k$-colouring and define \[ N_i = R_i(R_{i+1}( \cdots R_S(mS)) \cdots ), \] for $i = 1, \ldots, S$. By simply applying Ramsey's theorem at each step, it is clear that we can inductively construct sets $[N] \supset X_1 \supset X_2 \supset  \cdots \supset X_S$ with $|X_i| = N_{i+1}$, for $i \in [0,S-1]$ and $|X_S| = mS$, so that $F$ is constant on each of $X_{i}^{(1)}, \ldots, X_{i}^{(i)}$, for each $i \in [S]$.  
\paragraph{}
Now define a $k$-colouring $g$ of $[S]$ by $g(x) = F(A) \text{ for some } A \in X_S^{(x)}$. Since $F$ is constant on $X_S^{(x)}$, for $x \in [S]$, the value of $g(x)$ does not depend on the choice of $A$, and hence is well defined. Now apply the Finite Sums Theorem (Theorem~\ref{thm:finiteSumsThm}) to find $FS(a_1,\ldots,a_m) \subseteq [S]$, with $FS(a_1,\ldots, a_m)$ monochromatic with respect to the colouring $g$. To finish the proof, write $X_S = \{ x(1) < \cdots < x(lS) \}$ and choose $A_1$ to be the first $a_1$ elements of $X_S$, choose $A_2$ to be the next $a_2$ elements and so on. More formally, for $i \in [m]$ set 
\[ A_i = \{ x(a_{1} + \cdots + a_{i-1} +1 ) , x(a_{1} + \cdots + a_{i-1} + 2 ), \ldots ,x(a_{1} + \cdots + a_i ) \}. 
\] By construction, we have $\max A_i < \min A_{i+1}$. To see that $FU(A_1,\ldots,A_m)$ is monochromatic, simply note that 
$|A_{i_1} \cup A_{i_2} \cup \cdots \cup A_{i_k}| = a_{i_1} + \cdots + a_{i_l}$ for each $l \in [m]$, $1\leq i_1,< \cdots <i_l $, and hence   
\[ F(A_{i_1} \cup A_{i_2} \cup \cdots \cup A_{i_l}) = g(a_{i_1} + \cdots + a_{i_l}) = g(a_1). 
\]This completes the proof.
\end{proof}

We are now in a position to take a large step towards the proof of Theorem \ref{MainTheorem}. In what follows, we take our given colouring and construct a tremendous number of subsequences which interact in a favourable way. For the moment, we are not interested in explicitly constructing our desired pattern; we are only setting the scene. Later, in the proof of Theorem \ref{MainTheorem}, we shall carefully extract our FEP-set from the collection of sequences that we build here. 
\paragraph{}
In the proof below, it will be convenient to index the coordinates of $\mathbb{N}^{r+1}$ by elements of $[0,r]$ instead of elements of $[1,r+1]$, as is usual.

\begin{lemma} \label{structureBuildingLemma}
If $f : \mathbb{N} \rightarrow [k]$ is a $k$-colouring, then for all $l,r \in \mathbb{N}_0$ with $l \geq r$, we may construct
\begin{enumerate}
\item partitioned colourings $f_S : \mathbb{N}^{[0,r]}_0 \times \mathbb{N} \rightarrow [k]$, for each $S \in \mathcal{P}([0,r]) \setminus \{ \emptyset \}$ ;
\item derived colourings $\partial^1 f_S, \ldots, \partial^l f_S : \mathbb{N}^{[0,r]}_0 \times \mathbb{N} \rightarrow [k]$ for each $S \in \mathcal{P}([0,r]) \setminus \{ \emptyset \}$
;
\item sequences of integers $(d_0(M))_{M \in \mathbb{N}}, \ldots, (d_{r+1}(M))_{M \in \mathbb{N}}$ , where we set $d_0(M) = 1$ for all $M \in \mathbb{N}$;
\item sequences of integer vectors $(s_1(M))_{M \in \mathbb{N}}, \ldots, (s_{r+1}(M))_{M \in \mathbb{N}}$, where $s_j(M) \in \mathbb{N}_0^{[0,r]}$, $j\in [r+1]$, $M \in \mathbb{N}$, and we write $s_i(M) = (s_{i,0}(M), \ldots, s_{i,r}(M) )$, for $i \in [r+1]$;
\item \label{item:LimIndex} infinite sets of positive integers $C_0,\ldots,C_r$;
\item for each $i \in [0,r]$, and non-empty $S \subseteq [0,i]$, we have partitioned colourings \[ f_{S,i,M_0,\ldots,M_i} : \mathbb{N}^{[0,r]} \times \mathbb{N} \rightarrow [k], \] defined by
\[ f_{S,i,M_0,\ldots,M_i}(x_0,\ldots,x_r ; y) = f\left( 2 \star  \left( \sum_{a \in S}d_a(M_a)2^{s_{a+1,a}(M_{a+1}) + s_{a+2,a}(M_{a+2}) + \cdots + s_{i,a}(M_r) + x_a } + y \right) \right);\] 
\end{enumerate}
such that the following conditions hold. 
\begin{enumerate}
\item \label{synchronizedArrays} For each $i \in [0,r] $, the array 
\[ \left( f_S, \partial f_S, \ldots,  \partial^l f_S \right)_{\emptyset \not= S \subseteq [i]} 
\] is synchronized with profile $\left(d_{i+1}(M), s_{i+1}(M)\right)_{M \in \mathbb{N}}$;
\item \label{Convergent} for each $i \in [0,r]$, $S \subseteq [0,i]$, $t \in [0,i]$, and $M_{t+1},\ldots,M_{i} \in \mathbb{N}$, the limit 
\[ \lim_{M_t \in C_t} \lim_{M_{t-1} \in C_{t-1}} \cdots \lim_{M_0 \in C_0} \ f_{S,i,M_0,\ldots,M_i}
\] exists and 
\[ \lim_{M_i \in C_i} \lim_{M_{i-1} \in C_{i-1}} \cdots \lim_{M_0 \in C_0} \ f_{S,i,M_0,\ldots,M_i} = \partial^{i - \max{S}}f_S .\]
\end{enumerate}
\end{lemma}
\begin{proof}
We fix $l \in \mathbb{N}$ arbitrarily and apply induction on $r \leq l$. In the base case, $r = 0 $, we set $d_0(M_0) = 1$ for all $M_0 \in \mathbb{N}$ and consider the colouring 
\[ f_{\{0\},M_0}(x_0;y) =   f\left( 2^{d_0(M_0)2^{x_0} + y} \right) = f\left( 2^{2^{x_0} + y}  \right)  ,
\] which (trivially) tends to some limiting colouring as $M_0 \rightarrow \infty$, which we denote by $f_{ \{ 0 \} }$, in accordance with Condition~\ref{Convergent} of the induction hypothesis. Hence we have satisfied the Condition~\ref{Convergent} with the choice of $C_0 = \mathbb{N}$. To satisfy Condition~\ref{synchronizedArrays}, we apply Lemma~\ref{synchronizedArrayFact} to construct a sequence of derived colourings 
\[ f_{\{0\}}, \partial f_{\{0\}} , \ldots, \partial^l f_{\{0\}}
\] such that the above is a synchronized array with profile $(d_1(M), s_1(M))_{M \in \mathbb{N}}$. 
\paragraph{} Turning to the induction step, assume that the lemma holds for some $r-1 \leq l-1$: our aim is to show that it holds for $r$.
\paragraph{}
\emph{ Construction of the new colourings $\partial^i f_S$ } \\
We construct $\partial^if_S$ in two cases. In one case, if $r \not\in S$, we are essentially done: the induction hypothesis gives us what we need. In the other case, if $r \in S $, we find new sequences of integers such that the limiting colourings converge to something appropriate. To start with the easy case, $r \not\in S$, we use the induction hypothesis to construct $f_S, \partial f_{S}, \ldots, \partial^l f_{S}$ for all sets $\emptyset \not= S \subseteq [r-1]$. We define the new colourings $\partial^if_S$, for $\emptyset \not= S \subseteq [r-1]$, by simply joining a new ``dead'' variable $x_r$ to the $\partial^if_S$ that are granted by the induction hypothesis. That is, define 
\[ \widetilde{ \partial^i f}_S(x_0,\ldots,x_{r-1},x_{r};y) = \partial^i f_S (x_0, \ldots, x_{r-1};y), 
\] for all $x_0,\ldots,x_r,y \in \mathbb{N}$. From Fact \ref{adjoinADeadVariable}, we know that nothing is altered in adjoining this new variable, and so we take $\widetilde{ \partial^i f}_S$ as our $\partial^i f_S$ for all $\emptyset \not= S \subseteq [r-1]$, as in the statement of the lemma. As nothing essential is changed, we drop the tilde in the notation. 
\paragraph{}
This leaves us to construct $\partial^if_S$, where $r \in S$, and $i \in [0,l]$. To achieve this, we consider the following (totally new) collection of sequences, indexed by the non-empty subsets $S \subseteq [r]$, with $r \in S$ \\
\begin{equation} \label{Equation:KeySequences}\left( 2 \star \sum_{a \in S \setminus \{ r \}}d(M_a)2 \star \left( s_{a+1,a}(M_{a+1}) + \cdots + s_{r,a}(M_r) + x_a \right) + d(M_r)2^{x_r} + y \right)_{x_0,\ldots,x_r \in \mathbb{N}_0}. \end{equation} We shall take $f_S$ to be an appropriate limit point of $f$ along this sequence. More carefully, recall that the auxiliary colouring $f_{S,r,M_0,\ldots,M_r}(x;y)$ records the colouring $f$ along the sequence in (\ref{Equation:KeySequences}). Now apply Fact~\ref{Fact:Compactness} to an appropriate product colouring (or otherwise) to obtain the following.

\begin{claim} There exist infinite sets $C'_0,\ldots,C'_r$ such that $C'_0 \subseteq C_0, \ldots, C'_{r-1} \subseteq C_{r-1}$, $C'_r \subseteq \mathbb{N}$ and each of the limits  
\[ \lim_{M_t \in C'_t} \cdots \lim_{M_0 \in C'_0} f_{S,r,M_0,\ldots,M_r} 
\] exist, for $t \in [0,r]$, $S \subseteq [r]$, and $M_{t+1}, \ldots, M_r \in \mathbb{N}$. \qed
\end{claim} 
So we choose $C'_0,\ldots,C'_r$ to be our ``infinite sets of integers'' as in the statement of the Lemma and note that these sets guarantee the existence of the limits in Condition~\ref{Convergent} of the Lemma. Moreover, observe that each of the synchronized arrays $\left( f_S, \partial f_S, \ldots, \partial^lf_{S} :  \emptyset \not= S \subseteq [i]  \right)$, $i \in [r-1]$, will still be synchronized after restricting our indices to the $C'_i \subseteq C_i$, by Lemma~\ref{synchronizedArrayFact}.
\paragraph{}
We now define our new colourings $f_{S}$, for each $\emptyset \not= S \subseteq [0,r]$, with $r \in S$. We set  
\[  f_{S} = \lim_{M_r \in C'_r} \cdots \lim_{M_0 \in C'_0} f_{S,r,M_0,\ldots,M_r}, \] for each $\emptyset \not=  S \subseteq [0,r]$, with $r \in S$.
\paragraph{}
Finally, we apply Lemma~\ref{synchronizedArrayFact} to extend the synchronized array 
\[ \left( f_S, \partial f_S , \ldots, \partial^l f_S \right)_{\emptyset \not= S \subseteq [0,r-1]} 
\] to include all of the (newly constructed) colourings $f_{S}$ for all $\emptyset \not= S \subseteq [0,r]$, with $r \in S$. As a result, we obtain a synchronized array 
$\left( f_S, \partial f_S , \ldots, \partial^l f_S \right)_{\emptyset \not=S \subseteq [r]}$, with some (new) profile which we denote by $(d_{r+1}(M),s_{r+1}(M))_{M \in \mathbb{N}}$. We have thus constructed all the objects required by the lemma and have checked that they satisfy Condition~\ref{synchronizedArrays}. In what follows, we check that these objects satisfy Condition~\ref{Convergent}. 

\paragraph{} \emph{Check of Condition~\ref{Convergent} :} First note that we have already ensured that each of the limits appearing in Condition~\ref{Convergent} indeed exist. All that remains is to check that the identities hold. 
Now if $i \in [0,r-1]$, we have the desired identity by induction and this identity remains unaltered after we pass to a subsequence of the $M_0,\ldots,M_r$, as we have done above. Hence only the identities
\begin{equation} \label{Equation:SequenceAgreement0} \lim_{M_r \in C'_r} \cdots \lim_{M_0 \in C'_0} f_{S,i,M_0,\ldots,M_r} = \partial^{r - \max{S}}f_S 
\end{equation} remain to be checked for $\emptyset \not= S \subseteq [0,r]$. 
\paragraph{} 
As we shall only consider $M_i \in C'_i$ for $i \in [r]$, we shall drop the extra notation in the subscript of the limits. First, dealing with a trivial case, note that if $r \in S$, we have
\[ \lim_{M_r} \cdots \lim_{M_0} f_{S,r,M_0,\ldots,M_r} = f_S = \partial^{0}f_S ,
\] by construction of the colouring $f_S$. Proceeding to the non-degenerate case, we fix some $S \subseteq [r-1]$ and observe that, by induction, we have  
\[ \lim_{M_{r-1}} \cdots \lim_{M_0} f_{S,r-1,M_0,\ldots,M_{r-1}} = \partial^{r-1 - \max{S}}f_S. 
\]So for any fixed value of $M_r \in C'_r$, we have 
\begin{equation} \label{equ:limInMainLem} \lim_{M_{r-1}} \cdots \lim_{M_0} f_{S,r-1,M_0,\ldots,M_{r-1}}(x + s_r(M_r) ;y ) = \partial^{r-1 - \max{S}}f_S(x + s_r(M_r);y) = \partial^{r - \max{S}}f_S(x;y), \end{equation} where the last equation holds for any $x\in \mathbb{N}^{[0,r]},y \in \mathbb{N}$ with $y,|x|_{\infty} \leq M_r$. Now note that 
\[ f_{S,r,M_0,\ldots,M_r}(x;y) =  f_{S,r-1,M_0,\ldots,M_{r-1}}(x + s_r(M_r);y ) 
\] for $y,|x|_{\infty} \leq M_r$. Therefore and substituting this equation into~(\ref{equ:limInMainLem}) and taking the limit as $M_r \in C'_r$ tends to infinity  yields
\[ \lim_{M_{r}} \cdots \lim_{M_0} f_{S,r,M_0,\ldots,M_r} =  \partial^{r - \max{S}}f_S, 
\] as desired.
\paragraph{}
This completes the proof of Item \ref{Convergent} in the induction hypothesis and hence concludes the proof, by induction. 
\end{proof}
In what follows, we put the pieces together to prove Theorem \ref{MainTheorem}. 
\paragraph{}
\emph{Proof of Theorem \ref{MainTheorem}:} Let $f$ be a $k$-colouring of $\mathbb{N}$. By Lemma \ref{BasicRamseyLemma}, there exists some integer $N = N(k,m)$ such that every $k$-colouring of $\mathcal{P}([N])$ admits $A_1,\ldots , A_m \subseteq [N]$ such that $\max A_i < \min A_{i+1}$, $i \in [m-1]$, and such that $FU(A_1,\ldots,A_m)$ is monochromatic. Now, apply Lemma \ref{structureBuildingLemma} with $r = l = N$. We define a colouring $F : \mathcal{P}([N]) \rightarrow [k]$ by 
\[ F(S) = \partial^{N - \max S }f_S(0;0). 
\] In other words, we define $F(S)$ to be the colour attained by 
\begin{equation}\label{equ:formOfa_i} \lim_{M_N \in C_N} \lim_{M_{N-1} \in C_{N-1}} \cdots \lim_{M_0 \in C_0} f\left( 2 \star \sum_{a \in S}d_a(M_a)2^{s_{a+1,a}(M_{a+1}) + \cdots + s_{N,a}(M_r)} \right). 
\end{equation} Here, the values $s_{i,j}(M_k),d_i(M_i)$ are as in the statement of Lemma \ref{structureBuildingLemma}. 
\paragraph{}
Now apply Lemma \ref{BasicRamseyLemma} to find sets $A_1,\ldots,A_m \subseteq [N]$ such that $FU(A_1,\ldots,A_m)$ is monochromatic with respect to the colouring $F$ and such that $\max A_i < \min A_{i+1}$ for each $i \in [m-1]$. 
\paragraph{}
To understand what we have gained, for each $S \subseteq [N]$, let us define the numbers $a(S) = a\left(S; (M_a)_{N \geq a \geq \min {S}}\right)$ by 
\[ a(S) = 2 \star \sum_{a \in S}d_a(M_a)2^{s_{a+1,a}(M_{a+1}) + \cdots + s_{N,a}(M_N)} .
\] In this notation, we have that the $a(A_1),\ldots, a(A_m)$ form a monochromatic FP set for sufficiently large values $M_0 \in C_0, \ldots, M_N \in C_N$  (in the sense of Equation \ref{equ:formOfa_i}), as $a(A_i \cup A_j) = a(A_i)a(A_j)$, for each $i\not= j \in [m]$. In what follows, we show that the can pick the parameters $M_0,\ldots,M_N$ so that $a(A_1),\ldots,a(A_m)$ generate a monochromatic $FEP_W$ set. 
\paragraph{}
Before we turn to our main task, note that we may assume $W(A) \geq W(B)$, for all $A, B \in \mathbb{N}^{(\leq m)}$ with $A \supseteq B$, by possibly replacing $W$ with $W'(A) = \max\{W(B) : B \subseteq A \}$. The following claim simply makes explicit the various dependencies that we need on $M_N,\ldots, M_1$. We note that the proof of the claim below is trivial, as each integer $M_t$ depends only on $M_{t+1},\ldots,M_N,U_{t},\ldots,U_N$,
while each $U_t$ depends only on $U_{t+1},\ldots,U_N,M_{t+1},\ldots,M_N$.
\begin{claim} \label{ChoiceOfMt} There exists $(M_0,\ldots , M_N) \in C_0 \times \cdots \times C_N$, and auxiliary integers $U_0,\ldots,U_N \in \mathbb{N}$, so that, for each $t \in [0,N]$, $U_t$ satisfies the following conditions
\begin{enumerate} 
\item  $U_t  \geq \max FEP_W\left( \{ a(A_i) : A_i \subseteq [t+1,N] \} \right)$; 
\item  \label{ChoiceofMt:MtBiggerThanPrevExponents} 
$U_t \geq M_{t+1} + \sum_{N \geq i \geq t+1} |s_i(M_i)|_{\infty} + d_i(M_i)M_i $;
\item \label{ChoiceofMt:MtBiggerThanWeightedExponents} $ U_t \geq \sum_{N \geq i \geq t+1} |s_i(M_i)|_{\infty} + d_i(M_i)W(\{ a(S): S \subseteq [t+1,N] \})2\star \left( s_{i+1,i}(M_i) + \cdots + s_{N,i}(M_N) \right);$
\item $U_t \geq U_{t+1}$;
\end{enumerate}
and $M_t$ satisfies the following conditions.
\begin{enumerate}
\item \label{ChoiceofMt:Agreement} For each non-empty $S \subseteq [0,N]$, and $l \in [\max S ,N]$, we have
\[ \lim_{M_{t-1}} \cdots \lim_{M_0} f_{S,l,M_0,\ldots,M_{l}} (x;y) = \lim_{M_t} \lim_{M_{t-1}} \cdots \lim_{M_0} f_{S,l,M_0,\ldots,M_{l}} (x;y) 
\] for all $x,y$ satisfying $|x|_{\infty},y \leq U_t$;
\item \label{ChoiceofMt:MtBiggerThanUt} for each $M_t$ we have $M_t \geq U_t$.
\end{enumerate}
\end{claim} 
Actually, Conditions~\ref{ChoiceofMt:MtBiggerThanPrevExponents},\ref{ChoiceofMt:MtBiggerThanWeightedExponents} on $U_t$ are redundant, but we include them as they come at no added cost and they make explicit the precise conditions that we shall require later. Also notice that Condition~\ref{ChoiceofMt:Agreement} on the choice of $M_t$ implies that for each non-empty set $S \subseteq [0,N]$ and $l \in [\max S,N]$, we have that 
\begin{equation} \label{ChoiceofMt:agree2} f_{S,l,M_0,\ldots,M_l}(x;y) = \partial^{l-\max S} f_S(x;y), \end{equation} for all $x,y$ satisfying $|x|_{\infty},y \leq U_l$.
\paragraph{}
We now turn to check that the choice of $M_0,\ldots,M_N$ granted by Claim \ref{ChoiceOfMt} ensures that the integers $ a(A_1) = a(A_1;M_0,\ldots,M_N), \ldots, a(A_m) = a(A_m;M_0,\ldots,M_N) $ generate a monochromatic FEP set with weight $W$. As we have fixed the choices of $M_1,\ldots,M_M$, for $i \in [0,N]$, set 
\[ s_i = s_i(M) = (s_{i,0}(M_i),\ldots, s_{i,N}(M_i) ), \] 
\[ d_i = d_i(M_i) ,\]
\[ \tilde{s}_i = s_{i+1,i}(M_{i+1}) + s_{i+2,i}(M_{i+2}) + \cdots + s_{N,i}(M_i), \]
and \[ a_i = a(A_i).\]
\paragraph{}
We begin checking that the FEP set generated by $a_1,\ldots,a_m$ is monochromatic with the following observation, which recalls the relationship between the colourings $\partial^{j+1}f_S(x)$ and $\partial^jf_S(x)$
in the context of the proof. Formally, it follows from Condition \ref{synchronizedArrays} in Lemma \ref{structureBuildingLemma}.

\begin{observation} \label{ObsInProofOfTheorem} Let $S \subseteq [N]$. If $t$ is an integer such that $t \geq \max{S} + 1$, and $j \in [0,N-1]$, we have that 
\[ \partial^{j+1}f_S(x ; y) = \partial^jf_S(x + s_{t} + v d_t; y ), 
\]for all $y \in \mathbb{N}_0$, $x,v \in \mathbb{N}_0^{[0,N]}$ with $y, |x|_{\infty}, |v|_{\infty}  \leq M_t$.
\end{observation}
From this observation, we quickly arrive at the following claim. 

\begin{claim} \label{CascadeClaim} Suppose that $S \subseteq [N]$ and that $u,t$ are integers such that $u \geq 0$, $ t \in [N]$ and $t-u \geq \max S $, then
\[ \partial^{t - \max S} f_S(x;y) = \partial^{t - u - \max{S}} f_S\left( x + \sum_{i \in [t-u+1,t]} s_i + v_i d_i \ ; y \right),
\]  for all $y \in \mathbb{N}_0$ and $x,v_{t-u+1},\ldots,v_t \in \mathbb{N}_0^{[0,N]}$ with $y,|x|_{\infty} \leq M_t$, and $|v_i|_{\infty} \leq M_i$, for each $i \in [t-u+1,t]$. 
\end{claim}
\emph{Proof of Claim :} We prove this claim by induction on $u$. Notice that if $u = 0$, there is nothing to prove - we have an identity. So suppose we have shown 
\begin{equation} \label{equ:CascaseClaim} \partial^{t - \max S} f_S(x;y) = \partial^{t - u - \max{S}} f_S\left( x + \sum_{i \in [t-u+1,t]} s_i + v_i d_i \ ; y \right),\end{equation} for some $u \geq 0$ and for all $y \in \mathbb{N}$, $x,v_i \in \mathbb{N}^n_0$, $i \in [t-u+1,t]$, satisfying $y,|x|_{\infty} \leq M_t$, and $|v_i|_{\infty} \leq M_i$. Assume that $t-u-1\geq \max S$, otherwise there is nothing more to prove. Then apply Observation \ref{ObsInProofOfTheorem} to obtain 
\[ = \partial^{t - u - 1 - \max{S}} f_S\left( \left( x + \sum_{i \in [t-u+1,t]} s_i + v_i d_i \right) + s_{t-u} + v_{t-u}d_{t-u} \ ; y \right), \] for $y,|x|_{\infty} \leq M_t,|v_i|_{\infty} \leq M_i$, $i \in [t-u-1,t]$. Which holds by Observation \ref{ObsInProofOfTheorem}, as $y \leq M_t \leq M_{t-u}$ and each coordinate of the left argument of $\partial^{t-u-\max S}f_S$ at equation \ref{equ:CascaseClaim} is bounded by 
\[ \left( |x|_{\infty} + \sum_{i \in [t-u+1,t]} |s_i|_\infty + |v_i|_{\infty} d_i \right) \leq M_t + \sum_{i \in [t-u+1,t]} |s_i|_\infty + M_i d_i \leq M_{t-u}
,\] where the last inequality holds by Condition~\ref{ChoiceofMt:MtBiggerThanPrevExponents} on $U_t$ and Condition~\ref{ChoiceofMt:MtBiggerThanUt} on $M_t$, in Claim \ref{ChoiceOfMt}. Thus we are done, by induction. \qed
\paragraph{}

To show that $FEP_W(a_1,\ldots,a_m)$ is monochromatic, we again apply induction: we show that for each 
$\alpha \in FEP_W(a_1,\ldots,a_m) \setminus FP(a_1,\ldots,a_m)$ we can find some $\beta < \alpha$, $\beta \in FEP_W(a_1,\ldots,a_m)$
that has the same colour as $\alpha$. Hence, by induction, we shall see that each element of $FEP_W(a_1,\ldots,a_m)$ is given the same colour as 
some element in $FP(a_1,\ldots,a_m)$, which we already know to be monochromatic. For example, if we first know that $FP(a,b,c,d)$ is monochromatic in colour RED, say, we show that $a^bc^d \in FEP_W(a,d,c,d)$ is RED by showing that $a^bc^d$ has the same colour as $ac^d$, which we then show has the same colour as $ac \in FP(a,b,c,d)$. 

\begin{claim} For every $\alpha \in FEP(a_1,\ldots,a_m) \setminus FP(a_1,\ldots,a_m)$ there is some $\beta \in FEP(a_1,\ldots,a_m)$ such that 
$\beta < \alpha$ and $f(\alpha) = f(\beta)$.
\end{claim}
\emph{Proof of Claim :}
Let $\alpha \in FEP(a_1,\ldots,a_m) \setminus FP(a_1,\ldots,a_m)$ and write
\begin{equation}\label{Equ:FormOfAlpha} \alpha = \prod_{i \in B} a_i^{e_i}, 
\end{equation} where $B \subseteq [m]$, and $e_i \in FP_{[m]\setminus B, W}(x_{i+1},\ldots,x_m)$. For each $e_i$, we fix some representation of $e_i$, as product of powers of the elements $\{ a_j \}_{j \not\in B, j > i}$, and define $C_i$ to be the support of this product. That is, for $i \in B$, define $C_i \subset [m]$ to be a set for which  
\[ e_i = \prod_{j \in C_i} a_j^{p_{i,j}} ,
\] where $p_{i,j} \in [W\left(\{ a_{j+1},\ldots,a_m \} \right)]$, for each $j \in C_i$.
Now since $\alpha \not\in FP(a_1,\ldots,a_m)$, it must be that $\bigcup_{i \in B} C_i \not= \emptyset$. So we may set $j^* = \min \bigcup_{i \in B} C_i$. We define the number $\beta$ to be $\alpha$ ``with $a_{j^*}$ removed from the exponents''. More formally, if $i \in B$, let $e'_i \in \mathbb{N}$ be the integer satisfying $e'_i \cdot a_{j^*}^{p_{i,j^*}} = e_i$. Then define 
\[ \beta = \prod_{i \in B } a_i^{e'_i} .
\] Notice that $\beta \in FEP(a_1,\ldots,a_m)$ and $\beta < \alpha$. We show that $f(\alpha) = f(\beta) $. 

\begin{claim} \label{Claim:falphaEqualTofBeta}
Let $\alpha$, $\beta$ be as above, then $f(\alpha) = f(\beta)$.
\end{claim} 

\emph{Proof of Claim \ref{Claim:falphaEqualTofBeta}:} We partition the indexes of the bases, $B$, into two sets $B = B^- \cup B^+$, where $B^-$ consists of all indexes $<j^*$, while $B^+$ consists of all indexes $>j^*$. Of course, $j^*$ cannot appear in $B$, due to the form of $FEP$ sets. Note that $B^-$ is the collection of indices $i$ where $a_{j^*}$ \emph{may possibly} appear as an exponent of $a_i$ in $\alpha$. We partition $B^-$ into the indexes of the bases where $a_{j^*}$ \emph{does} appear as an exponent and the indexes where $a_{j^*}$ \emph{does not} appear as an exponent. That is, write $B^- = B^-_0 \cup B^-_1$, where $B^-_0 = \{ i : j^* \not\in C_i \}$ and $B^-_1 = \{ i : j^* \in C_i \}$. For simplicity, we also set $S = \bigcup_{i \in B^-} A_i $ and $A^* = A_{j^*}$.
\paragraph{} We now expand $\alpha,\beta$ in terms of the parameters $d_a$,$\tilde{s}_a$,
\begin{equation} \label{def:OfBeta} \beta = 2 \star \left( \sum_{i \in B } \sum_{a \in A_i } d_a2^{\tilde{s_a} + \rho_a } \right), \end{equation}
\begin{equation} \label{def:OfAlpha} \alpha = 2 \star \left( \sum_{i \in B } \sum_{a \in A_i } d_a2^{\tilde{s_a} + \tilde{\rho_a} } \right), 
\end{equation} where the particulars of the numbers $\rho_a$, $a \in [0,N]$ are not very important, but for concreteness we quickly note them. If $ a \not\in \bigcup_{i \in B } A_i $ then $\rho_a = 0 $. Otherwise, if $a \in A_i$ for some $i \in B$ (indeed this $i$ is unique as the $A_i$ are disjoint), we set 
\begin{equation} \label{equ:defofRho} \rho_a = \sum_{j \in C_i} p_{i,j} \sum_{b \in A_j } d_b 2^{\tilde{s_b}}. 
\end{equation} We define the numbers $\tilde{\rho_a}$, for $a \in [0,N]$, by setting $\tilde{\rho_a} = \rho_a$ if $a \not\in \bigcup_{i \in B^-_1} A_i$ and if $a \in A_i$, for some $i \in B^-_1$, we set
\begin{equation} \label{equ:defOfRhoTilde}
 \tilde{\rho_a} = \rho_a + p_{i,j^*} \sum_{b \in A^*} d_b2^{\tilde{s_b}} .
\end{equation} Note that $p_{i,j^*} \leq W\left( \{a_{j^*+1},\ldots,a_m\} \right) $, by the structure of FEP-sets. We now make note of the main facts about these numbers. If we set $\rho = (\rho_0,\ldots, \rho_N)$, and $\tilde{\rho} = (\tilde{\rho}_0, \ldots, \tilde{ \rho }_{N})$, we may express 
\begin{equation} \label{equ:ExpansionOfrho} \tilde{\rho} = \rho + \sum_{b \in A^*} d_bw^{(b)}, 
\end{equation} for some vectors $w^{(b)} \in \mathbb{N}_0^{[0,N]}$, $b \in A^*$, that are supported on $\bigcup_{i \in B^-_1} A_i \subseteq [0,N]$. Note that we have $|w^{(b)}|_{\infty} \leq M_b$, for each $b \in A^{*}$, by (\ref{equ:defOfRhoTilde}) and Condition~\ref{ChoiceofMt:MtBiggerThanPrevExponents} on $U_1,\ldots ,U_N$ and Condition~\ref{ChoiceofMt:MtBiggerThanUt} on $M_0,\ldots,M_N$ in Claim \ref{ChoiceOfMt}. 
It is also important to observe that 
\begin{equation} \label{boundOnRho}|\rho|_{\infty}  +  \sum_{i = \max A^* + 1}^{N} |s_i|_{\infty}   \leq  U_{\max A^*}, \end{equation}
\begin{equation} \label{equ:BoundonRhoTilde} |\tilde{\rho}|_{\infty} + \sum_{i = \max A^* + 1}^{N} |s_i|_{\infty}  \leq U_{ \max S }. \end{equation} 
The first inequality holds by the minimality of $j^*$ among $\bigcup_{i \in B} C_i$ and Claim \ref{ChoiceOfMt}, while the second inequality holds by (\ref{equ:defOfRhoTilde}), the fact that $\min A^* > \max S$ and Condition~\ref{ChoiceofMt:MtBiggerThanPrevExponents} in Claim \ref{ChoiceOfMt}.
\paragraph{} Now write $\rho = \rho^- + \rho^+$, where $\rho^-$ is supported on the coordinates with indexes in $S = \bigcup_{i \in B^-} A_i$ and 
$\rho^+$ is supported on coordinates with indexes in $\bigcup_{i \in B^+}A_i$. We also write $\tilde{\rho} = \tilde{\rho}^- + \tilde{\rho}^+$  ($= \tilde{\rho}^- + \rho^+$), in a similar way.  
\paragraph{}
We now proceed to show $f(\alpha) = f(\beta)$. We have,
\[ f(\beta) = f\left( 2 \star \left( \sum_{i \in B^-} \sum_{a \in A_i} d_a2^{\tilde{s}_a + \rho_a} + \sum_{i \in B^+} \sum_{a \in A_i} d_a2^{\tilde{s}_a + \rho_a} \right)\right) \] 
\begin{equation} \label{equ:FinalCalc} = f_{S,\max A^*,M_0,\ldots,M_{\max A^*}}\left( \rho^- + \sum_{\max A^*+1 \leq i \leq N } s_i \ \ ; \sum_{i \in B^+} \sum_{a \in A_i} d_a2^{\tilde{s}_a + \rho_a} \right) .\end{equation} Now set 
\[ \gamma = \sum_{i \in B^+} \sum_{a \in A_i} d_a2^{\tilde{s}_a + \rho_a}  
\] and note that
\begin{equation} \label{equ:BoundOnGamma}
\gamma \leq U_{\max A^*} \leq M_{\max A^*},
\end{equation}
by Condition~\ref{ChoiceofMt:MtBiggerThanUt} on $M_t$, and Condition~\ref{ChoiceofMt:MtBiggerThanPrevExponents} on $U_t$ in Claim \ref{ChoiceOfMt}, as $\gamma \leq \max FEP(\{a(S') : S' \subseteq [\max A^* + 1 ,N]\})$. This, along with the fact that the left argument in (\ref{equ:FinalCalc}) is bounded by $U_{\max A^*}$ (as noted at (\ref{boundOnRho})), implies that (\ref{equ:FinalCalc}) is
\begin{equation} \label{equ:FinalCalc2} = \partial^{\max A^* - \max S  }f_S\left( \rho^- + \sum_{\max A^*+1 \leq i \leq N } s_i \ \ ; \gamma \right), \end{equation}
by our selection of $M_0,\ldots,M_N$ at Claim \ref{ChoiceOfMt}. Now, apply Claim \ref{CascadeClaim} with the choice of $t = \max A^*$ and $u = \max A^* - \max S$ to learn that (\ref{equ:FinalCalc2}) is 
\begin{equation} \label{equ:FinalCalc3} = \partial^{ 0  }f_S\left( \rho^- + \sum_{\max A^*+1 \leq i \leq N } s_i  + \sum_{\max S + 1 \leq i \leq \max A^*} s_i + v_id_i \ \ ; \gamma \right),\end{equation}
for any choice of the $v_i \in \mathbb{N}_0^{[0,N]}$, $i \in [\max S +1, \max A^*]$,  satisfying $|v_i|_{\infty} \leq M_i$. We choose $v_i = w^{(i)}$ as in equation (\ref{equ:ExpansionOfrho}); then absorb the $w^{(i)}d_i$ terms into the $\rho$, in order to make $\tilde{\rho}$, as in (\ref{equ:ExpansionOfrho}). That is, we learn that (\ref{equ:FinalCalc3}) is
\[f_{S}\left( \tilde{\rho^-} + \sum_{\max S+1 \leq i \leq N } s_i  \ \ ; \gamma \right) = f_{S,\max{S},M_0,\ldots,M_{\max S}} \left( \tilde{\rho}^- + \sum_{ \max S +1 \leq i \leq  N} s_i \ \ ; \gamma \right) \]
\[ =f \left( 2 \star \left( \sum_{i \in B } \sum_{a \in A_i } d_a2^{s_a + \tilde{\rho_a} } \right)\right) = f(\alpha). 
\] To obtain the first equality, we have used that $\gamma \leq M_{\max A^*} \leq M_{\max S}$, by (\ref{equ:BoundOnGamma}), and that $\tilde{\rho^-} + \sum_{\max S+1 \leq i \leq N } |s_i|_{\infty} \leq U_{\max S}$, as noted at (\ref{equ:BoundonRhoTilde}), and then applied (\ref{ChoiceofMt:agree2}). The second to last equation follows by recalling the definition of $\gamma$; the last equation follows by recalling the expansion of $\alpha$, given above at (\ref{def:OfAlpha}). Thus, we conclude that $f(\alpha) = f(\beta)$. \qed
\paragraph{}
To finish the proof of the theorem, we simply apply induction on the elements of the $FEP$-set, as noted above. This completes the proof of the Claim and the proof of the Theorem. \qed \qed

\section{Non-partition regularity and a classification of first order systems} \label{sec:ClassificationResult}

In this section we introduce our ``$\log^*$-colourings''. These colourings will allow us to give many natural examples of exponential patters which are not partition regular. They also form the missing piece in the proof of our classification result for ``height-one'' systems, Theorem~\ref{ClassificationOfFirstOrderCompositions}. For an integer $n \geq 4$, we define the $n$\emph{th} $\log^*$-\emph{colouring}, to be the colouring (essentially) defined by ``colour $x$ with the least positive residue of $\log^* x$, when taken modulo $n$''. 
\paragraph{}
To be more precise, we declare all logarithms to be in \emph{base 2} and, for $k \in \mathbb{N}$, let $\log_{(k)} $ denote the $k$th iterated logarithm. Define the $\log^*$ function $L : \mathbb{N} \rightarrow \mathbb{N} \cup \{0\}$ as $L(x) = \min \left\lbrace k : \log_{(k)} x \leq 1 \right\rbrace$. It is easy to check that the function $L$ satisfies
\begin{enumerate}
\item $L$ is monotone increasing; \label{Lmonotone}
\item If $x \geq 2 $, then $L(\log x) + 1 = L(x)$; \label{composingLandl}
\item If $a,b \in \mathbb{N}$ and $a \leq b$ then $L(a + b) \leq L(b) + 1 $. \label{smalladditions}
\end{enumerate}
The restriction that these $\log^*$-colourings impose for partition regular patterns is made precise by the following result.  
\begin{lemma} \label{nobBiggerThana} For $r \in \mathbb{N}$, there exists a finite colouring $f : \mathbb{N} \rightarrow [r+3]$ admitting no monochromatic pair $b,a^b$ where $\log_{(r)} a \leq b$ and $a,b \not= 1$. 
\end{lemma}
\begin{proof}
Define the colouring of $f: \mathbb{N} \rightarrow [r+3]$ by first defining $f(1) = r+3$ and then, for all $x > 1$, we define $f(x)$ to be the least positive residue of $L(x)$, when taken modulo $r+2$. 
\paragraph{}
To check the desired property of $f$, let $a,b \in \mathbb{N}$. If $f(b)$ or $f(a^b)$ is $k+3$ we have $ a = b = 1$ and so we may assume that $a,b \geq 2$ and, in particular, that $a^b \geq 4$. We now show that $L(a^b) \geq L(b) +1 $. To do this, write 
\begin{equation} \label{equ:InMainNegResult} L(a^b) = L(\log_{(2)} a^b ) + 2.
\end{equation} This is possible, as  $a^b \geq 4$. Now the quantity at (\ref{equ:InMainNegResult}) is equal to
\[ L\left(\log\left( b\log a \right) \right) + 2 = L\left(\log b + \log_{(2)} a \right) + 2 \geq L(\log b) + 2 = L(b) + 1,
\] as desired. To bound $L\left(a^b\right)$ from above, we again write 
\[ L(a^b) = L\left( \log b + \log_{(2)} a \right) +2 \leq \max\{ L(b)-1 , L(\log_{(r)} a) + (r-2) \} + 3.
\] Now since $ \log_{(r)} a  \leq b $ the above is at most $L\left( b\right) + (r+1)$.
Putting these bounds together, allows us to conclude that $L(a^b) \in \{L(b)+1,L(b)+2,\ldots, L(b) + (r+1) \}$. It is therefore impossible for $L(a^b) \equiv L(b) \mod r+2$. This completes the proof.
\end{proof}

We now readily draw several conclusions. First, we note that if we take $r = 1$ in the above lemma, we obtain a $4$-colouring which admits no monochromatic quadruple of the form $a,b,a^b,b^a$. Going further, we see that this colouring, along with our main theorem (Theorem \ref{MainTheorem}), grants us a classification of ``height-one'' exponential systems, Theorem \ref{ClassificationOfFirstOrderCompositions}. Recall that this theorem says that one can find a monochromatic pattern of the form $\{x_1,\ldots,x_m \} \cup \{ x_i^{x_j} : (i,j) \in R\}$, where $R \subseteq [m]\times [m]$, in an arbitrary colouring if and only if $R$ does not contain a directed cycle - a set of the form $(y_1,y_2),\ldots,(y_{l-1},y_l),(y_l,y_1)$, where $y_1,\ldots,y_l,l \in [m]$. 
\paragraph{}
\emph{Proof of Theorem \ref{ClassificationOfFirstOrderCompositions} :} The positive direction follows immediately from the partition regularity of FEP-sets. Conversely, suppose that $R$ is a binary relation on $[m]$ with a directed cycle. Without loss, we may assume this cycle is $(1,2),\ldots,(l-1,l),(l,1)$, $ l \in [m]$. Apply Lemma~\ref{nobBiggerThana} to obtain a $4$-colouring $f$ admitting no monochromatic pair $a, a^b$ with $a < b $. If $R$ is partition regular we can find $x_1,\ldots,x_l$ such that $x_1,\ldots,x_l$ $x_1^{x_2},x_2^{x_3}, \ldots, x_l^{x_1}$ is monochromatic with respect to $f$. Hence $x_1 > x_2 > \cdots > x_l > x_1$, a contradiction. This completes the proof.\qed
\paragraph{}
Lemma~\ref{nobBiggerThana} can also be interpreted as giving a lower bound on the minimum integer $N(k)$, $k \in \mathbb{N}$, for which every $k$-colouring of $[N(k)]$ admits a monochromatic exponential triple. In particular, it shows that $N(k)$ is at least a tower of $2$s of height $k-3$.
\paragraph{}
Finally, we remark that the colouring given by Lemma~\ref{nobBiggerThana} and $r = 1$, also forbids the natural infinite extension of finite exponential sets. Indeed, we can even forbid patterns such as $\{ x_1,x_2,\ldots, \} \cup \{ x_1^{x_2},x_2^{x_3}, x_3^{x_4} \ldots \}$, where $x_i >1$ for all $i \in \mathbb{N}$.  
 
\section{Inconsistency of Exponential Triples and Schur Triples } \label{sec:consistency}
In this short section we note the peculiar relationship between exponentiation and addition in the context of partition regularity. We give an example of a $16$-colouring for which there is no monochromatic set of the form $x,y,x+y, a,b,a^b$, where $x,y,a,b \in \mathbb{N}$. Curiously, we require a single case of Fermat's Last Theorem for our result. As the case $n = 4$ has an elementary proof (which is actually due to Pierre de Fermat \cite{Azcel}), we appeal only to this case. A proof of this result appears in \cite{Edwards}, among many other sources.

\begin{theorem} \label{Theorem:FermatsLastTheorem} There are no integer solutions to 
\[ X^4 + Y^4 = Z^4 
\] with $X,Y,Z \not= 0$. \end{theorem}
We now prove the main result of this section, that $x,y,x+y$, $a,b,a^b$ are inconsistent partition regular patterns.

\emph{Proof of Theorem~\ref{thm:Inconsistent}: }
We define our colouring $f$ by first defining two auxiliary colourings $f_1,f_2 : \mathbb{N} \rightarrow \{0,1,2,3\}$ and then defining $f(x) =  (f_1(x), f_2(x)) $. We define $f_1(x)$ to simply be the remainder of $x$ modulo $4$. Now, for $x \in \mathbb{N} \setminus \{1 \}$, we define  
\[  l(x) = \max\left\lbrace b : x = a^b \text{ with } a,b \in \mathbb{N} \right\rbrace ,
\] and $l(1) = 0 $. Note that $l(x^y) = l(x)y$ for all $x,y \in \mathbb{N}$, with $x \not= 1$. We define $f_2(x)$ to be the remainder of $l(x)$ modulo $4$ for all $x \in \mathbb{N}$. 
\paragraph{}
Now suppose that $a,b, x,y \in \mathbb{N}$ are such that $x, y, x+y$, $a,b,a^b$, all receive the same colour. 
We learn that $x + y \equiv y \ (4)$ and thus $x \equiv 0 \ (4)$. It follows that $b \equiv 0 \ (4)$ and hence $l(a^b) = bl(a) \equiv 0 \ (4) $.
hence $l(x) \equiv l(y) \equiv l(x + y) \equiv 0 \ (4)$. So we may write $x,y,x+y$ in the form
\[ x = u^4 , y = v^4, x+y = w^4 ,
\] where $u,v,w$ are positive integers. But this is impossible as we have obtained the a non-trivial solution to the equation 
\[ u^4+ v^4 =  w^4 .\]  \qed
\section{Open Problems} \label{sec:Questions}
The $\log^*$-colourings, given in Section~\ref{ClassificationOfFirstOrderCompositions}, impose a strong ordering on the elements of an exponential partition regular pattern and can actually forbid any pattern where a base appears in its own exponent. For example, patterns such as $a,b,a^{b^a}$,  $a,b,a^{ba}$, or (trivially) $a,a^a$ are not partition regular. On the other hand, a quick glance at Theorem~\ref{MainTheorem} reveals that a ``height-two exponent'' has a great deal of freedom. For example, the elements $a\star \left( b^{c^c} \right)$, $a \star \left( b \star \left( {c^{c^c}} \right)\right)$ are included in the $FEP_W(a,b,c)$-set, for an appropriately chosen $W$. What we know much less about is the repetition of ``height-one'' exponents. The following question arises naturally. 

\begin{question} \label{q:abb}
Is the pattern $a,b,a^b,a^{b^b}$,  $a,b >1$ partition regular? 
\end{question}
We conjecture the answer to this question is ``no'' and support our guess with two observations. Firstly, we show we cannot affirmatively answer Question~\ref{q:abb} by only considering numbers of the form $2^x$, as we have done in Theorem~\ref{MainTheorem}. Secondly, we show that $a,b,a^{b^{b^b}}$ is not partition regular. The main trick that we use to resolve these questions is captured in the following lemma. We momentarily let $\{ x \}$ stand for the fractional part of a real number $x$ and, for a prime $p$ and integer $n >0$, we let $\nu_p(n)$ be the largest integer $k$ for which $p^k$ divides $n$.

\begin{lemma}\label{lemma:noDiffInSparseSets}
Let $  a_1 < a_2 < \cdots$ be a lacunary sequence of positive real numbers, that is with $\liminf \frac{a_{n+1}}{a_n} > 1$. Then there exists a finite colouring $f$ of $\mathbb{R}$ so that there is no monochromatic pair $x, x+a_n$ with $x \in \mathbb{R}, n\in \mathbb{N}$.
\end{lemma}
\begin{proof}
Let $\epsilon >0$ be such that $a_n > (1+\epsilon)a_n$, for all $n \in \mathbb{N}$, and choose $l \in \mathbb{N}$ to be such that $(1+\varepsilon)^l>4$. Now notice that if we partition $\{a_n\} = \bigcup_{i=0}^{l-1} \{ a_n : n = lm + i \}$ it is sufficient to find, for each individual part of the partition, a colouring that forbids the differences in that part; we then simply take the product of all of these colourings to forbid all differences in the union. Now notice that if we arrange each set in the above partition as an increasing sequence, each sequence satisfies the relation $b_{n+1} > (1+\varepsilon)^lb_{n}> 4b_{n}$, for $n \in \mathbb{N}$. So it is sufficient to show that if $\{b_n\}$ is a sequence satisfying $b_n > 4b_{n-1}$, $n \in \mathbb{N}$, we can find a colouring of $\mathbb{R}$ forbidding differences in $\{b_n\}$ 
\paragraph{}
To do this, we first find a real number $\alpha \in (0,1)$ so that $\{\alpha b_n\} \not\in [-1/4,1/4]$, for all $n$. Consider the sets 
\[ S_n = \left\lbrace \alpha \in [0,1] : \{ \alpha b_n \} \not\in [-1/4,1/4]  \right\rbrace .
\] We show that $\bigcap_{n} S_n \not= \emptyset$, by way of the following simple claim. If $I \subseteq \mathbb{R}/\mathbb{Z}$, is an interval
of length $1/2b_n$, then $I$ intersects $S_{n+1}$ in an interval of size $1/2b_{n+1}$. To see this, simply note that as $\alpha$ varies over $I$, $\{ \alpha b_{n+1} \}$ twice ranges over all of $[0,1]$. Hence $S_{n+1} \cap I$ contains an interval of length $1/2b_{n+1}$.
\paragraph{}
We now see that $\cap_n S_n \not= \emptyset$; choose $I_1 \subseteq S_1$ to be an interval of length $1/2a_1$, and then iteratively apply the lemma to construct $I_1 \supseteq I_2 \supseteq \cdots $ so that $\emptyset \not= \bigcap_{n<N} I \subseteq \bigcup_{n < N }S_n $ for all $N$. As the $I_i$ are closed, it follows that $\bigcap_n S_n \not= \emptyset$. We choose $\alpha $ to be a number in this intersection. 
\paragraph{}
For the construction of the colouring, partition $[0,1] = [0,1/4) \cup [1/4,1/2) \cup [3/4 ,1)$. Now define a colouring $f : \mathbb{R} \rightarrow [4]$ by setting $f(x)$ to be $i \in [4]$, if $\{\alpha x\} \in [(i-1)/4,i/4).$ Finally note that it is impossible for $x, x+a_n$ to be in the same colour class, as this would imply that $\{\alpha x\},\{\alpha x+ \alpha a_n\}$ lie in the same interval of length $1/4$, which is forbidden by the choice of $\alpha$.
\end{proof}

We now show that we can colour the subsequence $\{ 2^n \}$ so that there is no monochromatic triple of the form $a,a^{b^b}$.
\begin{prop}
One can finitely colour $\mathbb{N}$ so that there is no monochromatic pair $a, a^{b^b}$ with $a = 2^s,b =2^t$, for some $s,t \in \mathbb{N}$.
\end{prop}
\begin{proof} 
Notice that the sequence $\{ n2^n \}_n$ satisfies the conditions of Lemma~\ref{lemma:noDiffInSparseSets} and so we may find a finite colouring $f$ of $\mathbb{N}$ with no monochromatic pair of the form $m, m+n2^n$, with $m,n\in \mathbb{N}$. We then define the colouring $c$, by $c(x) = f(\nu_2(\nu_2(x)))$, for $x\in \mathbb{N} \setminus \{1\}$, and define $c(1)$ to be some colour which is distinct from all other colours. Now if we have $c(a) = c(a^{b^b})$, where $a,b$ are of the above form, we have $f(\nu_2(s)) = c(a) = c(a^{b^b}) = f\left( \nu_2(s) + t2^t \right)$, which contradicts our choice of $f$.\end{proof}

We also show that we may forbid monochromatic pairs of the form $a, a\star (b  \star (b \star b))$ in a finite colouring of $\mathbb{N}$.

\begin{prop} There exists a finite colouring of $\mathbb{N}$ such that there is no monochromatic pair of the form $a,a\star (b  \star (b \star b))$, $a,b >1$.
\end{prop}
\begin{proof} We use Lemma~\ref{lemma:noDiffInSparseSets} as in the previous proposition. For $n \in \mathbb{N}$, consider the sequence $\{ n^n\log_2 n \}_n$. Clearly this sequence satisfies the hypothesis of Lemma~\ref{lemma:noDiffInSparseSets} and so we may apply the lemma to obtain a colouring $f$ of $\mathbb{R}$ which forbids a monochromatic pair with a difference of this type. We then define a finite colouring of $\mathbb{N}\setminus \{1\}$ by $c(x) = f(\log_2\log_2(x))$, for $x \in \mathbb{N}\setminus \{1,2\}$, and by defining $c(2)$ to be a colour distinct from all other colours. Now if we have a monochromatic pair $a,a \star (b \star (b \star b))$ where $a,b>1$, it is clear that $a,b\not=2$. So we must have $f(\log_2\log_2(a)) = f( \log_2\log_2(a) + b^b\log_2(b) )$, contradicting the choice of $f$. 
\end{proof}

Another key property of FEP-sets is that they do not allow ``height-one exponents'' to also appear as a base in the same expression. For example, $a^b \cdot b$ is not, in general, contained in the set $FEP(a,b)$. However, we are unable to show that this is a necessary restriction.

\begin{question} Is $a,b,ab,a^b,a^b \cdot b$ partition regular, $a,b >1$?
\end{question}

Again, we conjecture the answer to the above is negative. Curiously, ``height-two exponents'' may reappear as a base in the same FEP-expression. Indeed, we have $a^{b^c} \cdot c \in FEP(a,b,c)$. We close with a incredibly simple-looking problem that we are unable to resolve.

\begin{question} \label{q:abatotheb+1}
Is $a,b,a^{b+1}, a,b>1$ partition regular ? 
\end{question}
This question has a natural multiplicative analogue.
\begin{question} Is $a,b,a(b+1)$ partition regular? 
\end{question}
We conjecture the answers to both questions are ``no'', but we are unable to resolve either. 

\section{Acknowledgements}
I should like to thank  B\'{e}la Bollob\'{a}s, Tom Brown, Veselin Jungi\'{c}, Imre Leader, and Micha\l{}  Przykucki for valuable discussions, support, and encouragement. I am especially indebted to Tom Brown and Veselin Jungi\'{c} for introducing the questions of Sisto to me.

\Addresses
\end{document}